\documentclass[12pt,a4paper,leqno]{amsart}
\usepackage[utf8]{inputenc}
\usepackage{amsmath,amsfonts,amscd,amssymb,latexsym,mathtools}
\usepackage[left=1.21in,right=1.21in,top=.975in,bottom=.975in]{geometry}

\usepackage{mathrsfs}
\usepackage[percent]{overpic} 

\usepackage[nobysame]{amsrefs}
\usepackage{enumitem}
\usepackage{comment}
\usepackage{soul}
\usepackage{tikz}
\usepackage{marginnote}
\usepackage{bbm}
\usepackage{cancel,slashed}
\usepackage[normalem]{ulem}

\usepackage{footmisc}
\usepackage{hyperref}
\hypersetup{
    colorlinks=true,
    linkcolor=blue,
    citecolor=green,}

\numberwithin{equation}{section}

\newtheorem{theorem}{Theorem}[section]
\newtheorem{lemma}[theorem]{Lemma}
\newtheorem{proposition}[theorem]{Proposition}
\newtheorem{corollary}[theorem]{Corollary}

\theoremstyle{definition}
\newtheorem{definition}[theorem]{Definition}
\newtheorem{example}[theorem]{Example}
\newtheorem{assumption}[theorem]{Assumption}

\theoremstyle{remark}
\newtheorem{remark}[theorem]{Remark}

\newcommand{\CC}{\mathbb{C}}
\newcommand{\NN}{\mathbb{N}}
\newcommand{\RR}{\mathbb{R}}

\newcommand{\ZZ}{\mathbb{Z}}

\newcommand{\calB}{\mathcal{B}}

\newcommand{\calD}{\mathcal{D}}

\newcommand{\calH}{\mathcal{H}}

\newcommand{\calR}{\mathcal{R}}

\newcommand{\calW}{\mathcal{W}}
\newcommand{\frakD}{\mathfrak{D}}
\newcommand{\frakM}{\mathfrak{M}}
\newcommand{\frakR}{\mathfrak{R}}

\newcommand{\ind}{\operatorname{ind}}

\newcommand{\dom}{\operatorname{dom}}

\newcommand{\aps}{{\rm APS}}

\newcommand{\rmc}{\mathrm{c}}

\newcommand{\upper}{\uppercase\expandafter}

\newcommand{\p}{\partial}
\newcommand{\pM}{{\p M}}

\newcommand{\ch}{\operatorname{ch}}
\newcommand{\Tch}{\operatorname{Tch}}

\newcommand{\End}{\operatorname{End}}
\newcommand{\rank}{\operatorname{rank}}

\newcommand{\tr}{\operatorname{tr}}
\newcommand{\Tr}{\operatorname{Tr}}

\renewcommand{\span}{\operatorname{span}}
\newcommand{\spf}{\operatorname{sf}}

\newcommand{\hatA}{\hat{A}}

\newcommand{\slaD}{\slashed{D}}
\newcommand{\slaS}{\slashed{S}}
\newcommand{\slaU}{\slashed{U}}

\newcommand{\tilg}{\tilde g}

\newcommand{\oeta}{\bar{\eta}}

\begin{document}

\normalsize

\title[Relative eta invariant and uniformly PSC]{Relative eta invariant and uniformly positive scalar curvature on non-compact manifolds}

\author[Pengshuai Shi]{Pengshuai Shi}
\address{School of Mathematics and Statistics, Beijing Institute of Technology, Beijing 100081, People's Republic of China}

\email{shipengshuai@bit.edu.cn, pengshuai.shi@gmail.com}



\thanks{The author was partially supported by the NSFC (grant no. 12101042) and Beijing Institute of Technology Research Fund Program for Young Scholars.}

\begin{abstract}
On complete non-compact manifolds with bounded sectional curvature, we consider a class of self-adjoint Dirac-type operators called Dirac--Schr\"odinger operators. Assuming two Dirac--Schr\"odinger operators coincide at infinity, by previous work, one can define their relative eta invariant. A typical example of Dirac--Schr\"odinger operators is the (twisted) spin Dirac operators on spin manifolds which admit a Riemannian metric of uniformly positive scalar curvature. In this case, using the relative eta invariant, we get a geometric formula for the spectral flow on non-compact manifolds, which induces a new proof of Gromov--Lawson's result about compact area enlargeable manifolds in odd dimensions. When two such spin Dirac operators are the boundary restriction of an operator on a manifold with non-compact boundary, under certain conditions, we obtain an index formula involving the relative eta invariant. This generalizes the Atiyah--Patodi--Singer index theorem to non-compact boundary situation. As a result, we can use the relative eta invariant to study the space of uniformly positive scalar curvature metrics on some non-compact connected sums.
\end{abstract}

\maketitle


\section{Introduction}\label{S:intro}

The index theory of Dirac operators has been proved to be a powerful means in the study of scalar curvature problems, ever since the seminal work of Atiyah--Singer \cite{AtSinger63} and Lichnerowicz \cite{Lich63}. In order to deal with more general situations, the original index theorem needs to be extended to non-compact manifolds. This was resolved satisfactorily in the influential work of Gromov--Lawson \cite{GromovLawson83}, where they developed the relative index theory, and made a great achievement in answering significant questions about positive scalar curvature. In recent years, there have been more advances in this active field. See the long article \cite{Gromov23Four} by Gromov for a nice and thorough exposition about results, techniques and problems in the subject of scalar curvature.

While the index of a Fredholm operator encodes information about the kernel of the operator (the difference of the dimensions of kernel and cokernel), the eta invariant of a self-adjoint Fredholm operator is a more sophisticated invariant measuring the spectral asymmetry of the operator (the regularized difference of the numbers of positive eigenvalues and negative eigenvalues). It originates in an index theorem of Atiyah--Patodi--Singer \cite{APS1} and later plays a role in the study of positive scalar curvature, such as in the work of Kreck--Stolz \cite{KreckStolz93}, Botvinnik--Gilkey \cite{BotviGilkey95eta-psc}, etc, where eta invariant is used to investigate the (moduli) space of positive scalar curvature metrics on closed manifolds.

Due to the spectral feature of the eta invariant, it shares more restrictive properties compared to the index. In particular, it is more difficult to generalize this notion to non-compact manifolds. In \cite{Shi22}, the author considered a relative version of the eta invariant, defined for a pair of Dirac-type operators acting on two non-compact manifolds which coincide at infinity. This definition has the advantage of requiring much milder conditions than defining individual eta invariants on non-compact manifolds. The author also studied some properties of the relative eta invariant in \cite{Shi22}.

In this paper, we take a closer look at the relative eta invariant, with the focus on its geometric implications. To this end, we will be mainly concerned about the (twisted) spin Dirac operators on non-compact Riemannian spin manifolds with uniformly positive scalar curvature (PSC for short). In this case, the relative eta invariant satisfies a gluing formula (without mod $\ZZ$), and induces a geometric formula for the spectral flow. Moreover, we can get an APS-type index formula on manifolds with non-compact boundary. These make it possible to investigate uniformly PSC metrics on some non-compact spin manifolds.

\subsection{Summary of the main results}\label{SS:main res}

Let $\calD$ be a formally self-adjoint Dirac-type operator on a non-compact complete Riemannian manifold $M$ without boundary. We call $\calD$ a \emph{Dirac--Schr\"odinger operator} if $\calD^2$ is a Schr\"odinger operator with the potential being uniformly positive at infinity (cf. Definition~\ref{D:Dirac-Schro}). For two Dirac--Schr\"odinger operators $\calD_0$ and $\calD_1$ on two respective manifolds $M_0$ and $M_1$ with bounded sectional curvature, if they coincide at infinity, then their \emph{relative eta invariant}, denoted by $\eta(\calD_1,\calD_0)$, can be defined from a heat operator regularization (cf. Proposition~\ref{P:rel-eta}). It has been proved in \cite{Shi22} that the relative eta invariant satisfies a mod $\ZZ$ gluing formula. When the operators are invertible, we show in Theorem~\ref{T:gluing} that this formula is actually a real equality (which is well-known for eta invariant on compact manifolds).

The eta invariant is closely related to the spectral flow. On non-compact spin manifolds, we can get the following geometric formula computing the spectral flow (cf. Subsection~\ref{SS:Chern-Simons sp-flow }).

\begin{theorem}[cf. Theorem~\ref{T:sp flow-trans}]\label{T:intro-1}
Suppose that $M$ is an odd-dimensional non-compact Riemannian spin manifold which admits a complete metric $g$ of uniformly PSC and bounded sectional curvature. For $r\in[0,1]$, let $\nabla_r^F$ be a linear path of connections connecting two flat connections $\nabla_0^F$ and $\nabla_1^F$ on a Hermitian vector bundle $F$ over $M$ such that they coincide at infinity. Let $\slaD_{E,r}$ be the associated family of twisted spin Dirac operators on the twisted spinor bundle $E=\slaS\otimes F$. Then
\[
\begin{aligned}
\spf(\slaD_{E,r})_{[0,1]}=&\int_M\hat{A}(M,g)\Tch(\nabla_0^F,\nabla_1^F) \\
&+\frac{1}{2}\big(\eta(\slaD_{E,1},\slaD_{E,0})+\dim\ker\slaD_{E,1}-\dim\ker\slaD_{E,0}\big),
\end{aligned}
\]
where $\hat{A}(M,g)$ is the $\hat{A}$-genus form of $(M,g)$, and $\Tch(\nabla_0^F,\nabla_1^F)$ is the Chern--Simons form associated to $\nabla_0^F,\nabla_1^F$.
\end{theorem}

This formula can be used to give a new proof of Gromov--Lawson's result that there does not exist a PSC metric on a compact area enlargeable manifold in the odd-dimensional case.

One of the natural questions about relative eta invariant is whether it appears as a boundary term of an APS index formula on a manifold with non-compact boundary. This was studied in \cite{Shi22} for strongly Callias-type operators in some special cases. The reason to consider strongly Callias-type operators there is that they have discrete spectra. Thus as Braverman and the author did in \cite{BrShi21,BrShi21-2}, one can talk about the APS boundary condition like the compact case. In this paper, we will examine spin Dirac operators on manifolds endowed with a uniformly PSC metric. In this case, the boundary operator is only Fredholm and can have continuous spectrum away from zero. As observed in \cite{APS1}, the index with APS boundary condition can be identified with an $L^2$-index on the manifold obtained by attaching a cylinder to the boundary (called \emph{elongation}). From this point of view, we define an APS-type index on a manifold with non-compact boundary to be the $L^2$-index on the elongation of the original manifold.

Let $M$ be an odd-dimensional non-compact Riemannian spin manifold without boundary, admitting a complete metric $g$ of uniformly PSC and bounded sectional curvature. Let $\frakR_\infty^+(M,g)$ denote the space of uniformly PSC metrics on $M$ which coincide with $g$ at infinity. For $g_0,g_1\in\frakR_\infty^+(M,g)$, our index formula is for a special kind of manifolds called \emph{cobordism} between $g_0$ and $g_1$. Basically, it is a manifold whose boundary constitutes two components, one of which is isometric to $(M,g_0)$, and the other is isometric to $(M,g_1)$ (cf. Subsection~\ref{SS:space psc}). Our index formula is formulated as follows.

\begin{theorem}[cf. Theorem~\ref{T:cobor rel-eta}]\label{T:intro-2}
Let $(W,g_W)$ be a cobordism between two metrics $g_0,g_1\in\frakR_\infty^+(M,g)$ of bounded sectional curvature, and let $(\widetilde{W},\tilde{g})$ be the elongation of $(W,g_W)$. Let $\slaD_{\slaS_0},\slaD_{\slaS_1}$ and $\slaD_{\slaS_{\widetilde{W}}}$ be the corresponding spin Dirac operators on $(M,g_0),(M,g_1)$ and $(\widetilde{W},\tilde{g})$, respectively. Then
\[
\ind\slaD_{\slaS_{\widetilde{W}}}^+=\int_W\hatA(W,g_W)+\frac{1}{2}\eta(\slaD_{\slaS_1},\slaD_{\slaS_0}).
\]
\end{theorem}

The proof of this theorem combines an ``$L^2$-index equals spectral flow'' result on a cylinder and the relation of spectral flow with relative eta invariant obtained in Section~\ref{S:sp flow}.

Having this index theorem, one can get some information about the space of uniformly PSC metrics on some high-dimensional non-compact manifolds. This is a question that was seldom considered before. What we can handle in this paper is the space $\frakR_\infty^+(M,g)$ related to connected sums. To be precise, let $N$ be an odd-dimensional ($\ge5$) closed Riemannian spin manifold which admits a PSC metric. Assume $N$ has a non-trivial finite fundamental group and satisfies an extra condition when $\dim N\equiv1$ mod 4 (see Subsection~\ref{SS:PSC conn-sum odd}). It was shown by Botvinnik and Gilkey in \cite{BotviGilkey95eta-psc} that $N$ admits infinitely many PSC metrics which are not cobordant to each other in the space of PSC metrics. We generalize this result to non-compact connected sums in both odd- and even-dimensional situations. (The even-dimensional case generalizes result of Mrowka--Ruberman--Saveliev \cite{MRS16}.)

\begin{theorem}[cf. Theorems~\ref{T:PSC conn-sum odd} and \ref{T:PSC conn-sum even}]\label{T:intro-3}
Let $N$ be as above. Let $N'$ and $X$ be two non-compact Riemannian spin manifolds without boundary such that $\dim N'=\dim N=\dim X-1$. Suppose $N'$ (resp. $X$) admits a complete metric $h'$ (resp. $\gamma$) of  uniformly PSC and bounded geometry. Set
\[
M_0=N\#N',\quad M_1=(N\#N')\times S^1,\quad M_2=(N\times S^1)\#X.
\]
Then
\begin{enumerate}[leftmargin=*]
\item there exist infinitely many uniformly PSC metrics which are not cobordant to each other in $\frakR_\infty^+(M_0,h')$; \label{IT:intro-3-1}
\item $\frakR_\infty^+(M_1,h'+ds^2)$ and $\frakR_\infty^+(M_2,\gamma)$ both have infinitely many path components. \label{IT:intro-3-2}
\end{enumerate}
\end{theorem}

Note that \ref{IT:intro-3-1} implies that $\frakR_\infty^+(M_0,h')$ has infinitely many path components as well.
This theorem also induces similar result on manifolds with boundary (see Theorem~\ref{T:PSC bdry}).

We expect the relative eta invariant to have broader applications in geometric and topological problems. In order for that to happen, the restriction of bounded sectional curvature (or bounded geometry) should be removed. In the future, we will work on a more general notion of relative eta invariant, and explore its consequences for a wider range of scenarios.

\subsection{Organization of the paper}\label{SS:orga}

This paper is organized as follows. In Section~\ref{S:glu-releta}, we introduce the relative eta invariant for Dirac--Schr\"odinger operators and derive a gluing formula. In Section~\ref{S:sp flow}, we use the relative eta invariant to compute the spectral flow, prove Theorem~\ref{T:intro-1}, and discuss its application. In Section~\ref{S:ind psc}, we consider equivalence relations on the space of uniformly PSC metrics which coincide at infinity and prove Theorem~\ref{T:intro-2}. Section~\ref{S:psc conn-sum} is about some non-triviality results in terms of the connectedness of the space of uniformly PSC metrics on connected sums, where Theorem~\ref{T:intro-3} is proved.



\section{Relative eta invariant and a gluing formula in the invertible case}\label{S:glu-releta}

In this section, we first review the notion of relative eta invariant for a pair of Dirac-type operators called Dirac--Schr\"odinger operators. Then we give a gluing formula for the relative eta invariant that will be used in later sections.

\subsection{Dirac--Schr\"odinger operators}\label{SS:Dirac-Sch}

Let $S\to M$ be a Hermitian vector bundle over a complete Riemannian manifold $M$ (with or without boundary). We call $S$ a \emph{Dirac bundle} if there is a Clifford multiplication $\rmc(\cdot):TM\to\End(S)$ which is skew-Hermitian and satisfies $\rmc(\cdot)^2=-|\cdot|^2$, and a Hermitian connection $\nabla^S$ that is compatible with $\rmc(\cdot)$ (cf. Lawson--Michelsohn \cite[\S II.5]{LawMic89}). The (compatible) \emph{Dirac operator} is a first-order differential operator acting on sections of a Dirac bundle, defined by
\[
D:=\sum_{i=1}^{\dim M}\rmc(e_i)\nabla_{e_i}^S,
\]
where $e_1,\dots,e_{\dim M}$ is a local orthonormal frame of $TM$, and we use the Riemannian metric to identify $TM$ with $T^*M$. Usually, we view $D$ as an unbounded operator in $L^2(S)$ with domain $\dom D\subset L^2(S)$ consisting of all $s\in L^2(S)$ such that $Ds\in L^2(S)$.

In general, let $F$ be another Hermitian vector bundle with Hermitian connection $\nabla^F$. We can extend the Clifford multiplication to $S\otimes F$ by acting as identity on $F$ and form a connection
\[
\nabla^{S\otimes F}=\nabla^S\otimes 1+1\otimes\nabla^F
\]
on $S\otimes F$, so that $E:=S\otimes F$ is again a Dirac bundle. In this case one can define a twisted Dirac operator on $E$.

It is well known that a Dirac operator is formally self-adjoint. From the definition, Dirac operator is just about the square root of a Laplacian. To be precise, one has the Weitzenb\"ock formula (cf. Lawson--Michelsohn \cite[\S II.8]{LawMic89})
\begin{equation}\label{E:Weitz for}
D^2=\nabla^*\nabla+\calR,
\end{equation}
where $\nabla^*\nabla$ is the connection Laplacian on $E$ and $\calR$ is a bundle map which comes from the curvature transformation of $E$.

Operators which have the same principal symbol as a Dirac operator are called \emph{Dirac-type operators}. They can be written as a compatible Dirac operator plus a bundle map on $E$ (called a \emph{potential}). In this article, we will be focusing on a special kind of Dirac-type operators. 

\begin{definition}\label{D:Dirac-Schro}
Let $\calD:\dom\calD\to L^2(E)$ be a formally self-adjoint Dirac-type operator. We call $\calD$ a \emph{Dirac--Schr\"odinger operator} if $\calD^2-\nabla^*\nabla$ is a bundle map which has a uniformly positive lower bound outside a compact subset $K\Subset M$.
Here, $K$ is called an \emph{essential support} of the operator $\calD$.
\end{definition}

\begin{remark}\label{R:Dirac-Schro}
When $M$ is a non-compact manifold without boundary, one can easily see that a Dirac--Schr\"odinger operator is invertible at infinity.\footnote{In this paper, an operator $\calD$ is said to be \emph{invertible} (or \emph{coercive}) if there exists some constant $C>0$ such that $\|\calD s\|^2_{L^2(E)}\ge C\|s\|^2_{L^2(E)}$ for each $s\in\dom\calD$. $\calD$ being invertible at infinity means that the above estimate holds for each $s\in\dom\calD$ supported outside a compact subset.} Therefore, it has zero in its discrete spectrum, i.e., the operator is Fredholm. 
\end{remark}

\begin{example}\label{Eg:Dirac-Schro-1}
Let $(M,g)$ be a Riemannian spin manifold and $\slaS\to M$ be the (complex) spinor bundle endowed with its canonical Hermitian connection induced by the Levi--Civita connection. There is a spin Dirac operator (or Atiyah--Singer operator) $\slaD_\slaS$ acting on sections of $\slaS$ (see Berline--Getzler--Vergne \cite[Chapter 3]{BeGeVe}, Lawson--Michelsohn \cite[Chapter II]{LawMic89}). When $\dim M$ is even, there is a $\ZZ_2$-grading $\slaS=\slaS^+\oplus\slaS^-$, with respect to which $\slaD_\slaS$ is $\ZZ_2$-graded
\[
\slaD_\slaS=\left(
\begin{matrix}
0 & \slaD_\slaS^- \\
\slaD_\slaS^+ & 0
\end{matrix}
\right),
\]
so that $\slaD_\slaS^+$ and $\slaD_\slaS^-$ are formal adjoint to each other.

For spin Dirac operators, one has the Lichnerowicz formula
\[
\slaD_\slaS^2=(\nabla^\slaS)^*\nabla^\slaS+\frac{\kappa}{4},
\]
where $\kappa$ is the scalar curvature associated to $g$. Suppose $g$ is a metric of \emph{uniformly} positive scalar curvature outside a compact subset. Then $\slaD_\slaS$ is a Dirac--Schr\"odinger operator.

In general, let $\slaD_{\slaS\otimes F}$ be a twisted spin Dirac operator. Then
\begin{equation}\label{E:Lich for-1}
\slaD_{\slaS\otimes F}^2=(\nabla^{\slaS\otimes F})^*\nabla^{\slaS\otimes F}+\frac{\kappa}{4}+\calR^F,
\end{equation}
where $\calR^F=\sum_{i<j}\rmc(e_i)\rmc(e_j)R^F(e_i,e_j)$, and $R^F=(\nabla^F)^2$ denotes the curvature operator. If $\frac{\kappa}{4}+\calR^F$ is uniformly positive outside a compact subset (which automatically holds when $F$ is a flat bundle), then $\slaD_{\slaS\otimes F}$ is again a Dirac--Schr\"odinger operator.
\end{example}

\begin{example}\label{Eg:Dirac-Schro-2}
Let $\calD=D+\Psi$, where $D$ is a compatible Dirac operator and $\Psi\in\End(E)$ is a self-adjoint bundle map. $\calD$ is called a \emph{Callias-type operator} if roughly speaking, $\calD^2-D^2$ is a bundle map which has a uniformly positive lower bound outside a compact subset. In this case, by choosing a suitable potential $\Psi$ which satisfies certain growth condition, one can make $\calD$ a Dirac--Schr\"odinger operator.
\end{example}

\subsection{The relative eta invariant}\label{SS:rel-eta}

The eta invariant was first introduced by Atiyah--Patodi--Singer \cite{APS1} in their celebrated APS index formula. It measures the spectral asymmetry of a self-adjoint operator on a closed manifold. On a non-compact manifold, the eta invariant usually cannot be defined. In \cite{Shi22}, the author establishes the notion of relative eta invariant on two non-compact manifolds which coincide at infinity.

For $j=0,1$, let $\calD_j$ be a Dirac--Schr\"odinger operator acting on sections of $E_j$ over a non-compact complete Riemannian manifold $M_j$. Suppose that outside two compact subsets $K_0\Subset M_0$ and $K_1\Subset M_1$, the manifolds $M_0$ and $M_1$ are isometric, the bundles $E_0$ and $E_1$ are isomorphic so that $\calD_0$ and $\calD_1$ coincide at infinity.

\begin{proposition}[\cite{Shi22}]\label{P:rel-eta}
Let $M_0$ and $M_1$ be two non-compact complete Riemannian manifolds without boundary, $\calD_0$ and $\calD_1$ be two Dirac--Schr\"odinger operators on $(M_0,E_0)$ and $(M_1,E_1)$, respectively which coincide at infinity. Consider the \emph{relative eta function}
\begin{equation}\label{E:rel-eta func}
\eta(s;\calD_1,\calD_0)= 
\frac{1}{\Gamma((s+1)/2)}\int_0^\infty t^{(s-1)/2}\Tr\big(\calD_1e^{-t\calD_1^2}-\calD_0e^{-t\calD_0^2}\big)dt.
\end{equation}
If $M_0$ and $M_1$ have bounded sectional curvature, then $\eta(s;\calD_1,\calD_0)$ is well-defined when $\Re(s)$ is large and admits a meromorphic continuation to the whole complex plane. Moreover, it is regular at $s=0$.
\end{proposition}

Due to Proposition~\ref{P:rel-eta}, the following definition can be made.

\begin{definition}\label{D:rel-eta}
In view of Proposition~\ref{P:rel-eta}, the \emph{relative eta invariant} associated to $\calD_0$ and $\calD_1$ is defined to be $\eta(0;\calD_1,\calD_0)$. For simplicity, we denote it by $\eta(\calD_1,\calD_0)$. 

In the case that the relative eta function may not be regular at $s=0$ (for example, on manifolds with boundary, see Lemma~\ref{L:rel-eta bdry}), the relative eta invariant is defined to be the constant term in the Laurent expansion of the relative eta function at $s=0$.

The \emph{reduced relative eta invariant} is defined to be
\[
\xi(\calD_1,\calD_0):=\frac{1}{2}\left(\eta(\calD_1,\calD_0)+\dim\ker\calD_1-\dim\ker\calD_0\right).
\]
\end{definition}

The relative eta invariant, as the name suggests, can be thought of as the difference of two individual eta invariants. In particular,
\begin{equation}\label{E:rel-eta prop}
\eta(\calD_0,\calD_0)=0,\qquad
\eta(\calD_2,\calD_1)+\eta(\calD_1,\calD_0)=\eta(\calD_2,\calD_0).
\end{equation}

\subsection{Relative eta invariants on manifolds with boundary}\label{SS:rel-eta bdry}

Like eta invariant, the relative eta invariant also possesses a gluing formula. In this subsection, we recall the basic setting as in Kirk--Lesch \cite{KirkLesch04} and Shi \cite{Shi22}. Let $\calD_0$ and $\calD_1$ be two Dirac--Schr\"odinger operators on $(M_0,E_0)$ and $(M_1,E_1)$, respectively, as above. 
Let $\Sigma_0\cong\Sigma_1\cong\Sigma$ be a common closed hypersurface of $M_0$ and $M_1$ with trivial normal bundle. Assume $\Sigma_j$ ($j=0,1$) lies outside the compact set $K_j$ of last subsection,
and $(M_j,E_j)$ has product structure near $\Sigma$ so that $\calD_j$ has the form $\calD_j=\rmc(\nu)(\p_u+\calB)$ in a tubular neighbourhood $[-\varepsilon,\varepsilon]\times\Sigma$ of $\Sigma$, where $\nu$ is a unit normal vector field along $\Sigma$ and $\calB$ is a self-adjoint Dirac-type operator on $\Sigma$. Note that since $\calD_j$ is formally self-adjoint, one always has $\rmc(\nu)\calB=-\calB\rmc(\nu)$. Now let $M_j^{\rm cut}$ denote the manifold with boundary obtained by cutting $M_j$ along $\Sigma$. Then under the identification
\[
L^2(E_j|_{\p M_j^{\rm cut}})=L^2(E_j|_\Sigma)\oplus L^2(E_j|_\Sigma),
\]
the resulting operator on $M_j^{\rm cut}$, which for simplicity, we adopt a slight abuse of notation and still denote by $\calD_j$, can be written as
\begin{equation}\label{E:decomp of D}
\calD_j=\left(
\begin{matrix}
\rmc(\nu) & 0 \\
0 & -\rmc(\nu)
\end{matrix}
\right)\left(\p_u+
\left(
\begin{matrix}
\calB & 0 \\
0 & -\calB
\end{matrix}
\right)\right)
=:\tilde{\rmc}(\nu)\left(\p_u\,+\,\widetilde{\calB}\right),
\end{equation}
near the boundary $\pM_j^{\rm cut}=\Sigma\oplus\Sigma$.

\begin{assumption}\label{A:bounded geom}
For $j=0,1$, assume that
\begin{enumerate}
\item $(M_j,E_j)$ has \emph{bounded geometry} of order $m>\dim M_0/2$, i.e., $M_j$ has uniformly positive injectivity radius, and the curvature tensors of $M_j$ and $E_j$ along with their covariant derivatives up to order $m$ are uniformly bounded. \label{IT:bounded geom-1}

\item There exists a constant $C>0$ such that for all integers $1\le k\le m$ and $s\in\dom(\calD_j^k)$,
\[
 \Vert\calD_j^ks\Vert_{L^2(E_j)}^2+\Vert s\Vert_{L^2(E_j)}^2\ge C\,\Vert D_j^ks\Vert_{L^2(E_j)}^2,
\]
where $D_j$ is the corresponding compatible Dirac operator (i.e., without potential), $\calD_j^k$ and $D_j^k$ denote the $k$-th power of $\calD_j$ and $D_j$, respectively. \label{IT:bounded geom-2}
\end{enumerate}
\end{assumption}

\begin{remark}\label{R:bounded geom}
If $\calD_j$ is the spin or twisted spin Dirac operator on a spin manifold as in Example~\ref{Eg:Dirac-Schro-1}, then it is itself a compatible Dirac operator. In this case, the estimate in Assumption~\ref{A:bounded geom}.\ref{IT:bounded geom-2} is automatically satisfied.
\end{remark}

On $M_j^{\rm cut}$, we impose two natural boundary conditions on $\calD_j$. One is called the \emph{continuous transmission boundary condition}, which corresponds to the domain
\[
\dom(\calD_{j,\Delta})=\left\{s\in\dom(\calD_{j,\max}):\,s|_{\p M_j^{\rm cut}}=(f,f)\in L^2(E_j|_\Sigma)\oplus L^2(E_j|_\Sigma)\right\},
\]
where $\dom(\calD_{j,\max})$ is the domain of the maximal extension of $\calD_j$ on $M_j^{\rm cut}$ (cf. B\"ar--Ballmann \cite[Example 7.28]{BaerBallmann12}).
Equivalently, let
\[
P_\Delta=\frac{1}{2}\left(
\begin{matrix}
1 & -1 \\
-1 & 1
\end{matrix}\right)
\]
be the continuous transmission projection. The domain can also be written as
\[
\dom(\calD_{j,\Delta})=\left\{s\in\dom(\calD_{j,\max}):\,P_\Delta(s|_{\p M_j^{\rm cut}})=0\right\}.
\]
The other one is called the \emph{Atiyah--Patodi--Singer boundary condition}. Suppose there exists a Lagrangian subspace $L\subset\ker\calB$, namely, $L$ is a subspace of $\ker\calB$ satisfying $\rmc(\nu)L=L^\perp\cap\ker\calB$.\footnote{By the fact that $\rmc(\nu)$ anti-commutes with $\calB$, the action of $\rmc(\nu)$ induces a skew-Hermitian form on $\ker\calB$ by $\omega(s,s'):=\langle s,\rmc(\nu)s'\rangle_{L^2(E_j|_{\Sigma})}$. A Lagrangian subspace exists if and only if the $\sqrt{-1}$ and $-\sqrt{-1}$ eigenspaces of $\rmc(\nu)$ in $\ker\calB$ have the same dimension. In particular, if $(\Sigma,\calB)$ bounds $(X,\calD)$, where $X$ is a compact manifold with boundary and $\calD$ is a formally self-adjoint Dirac-type operator on $X$, then by the cobordism invariance of the index, such Lagrangian subspaces always exist. See \cite[Section~2]{KirkLesch04} for details.\label{FN:Lagran}}
Let $\Pi^+(\calB)$ be the spectral projection onto the eigenspaces corresponding to positive eigenvalues of $\calB$, and $P_L$ be the $L^2$-orthogonal projection onto $L$. Denote
\[
P_\aps(L):=\left(
\begin{matrix}
P^+_L & 0 \\
0 & 1-P^+_L
\end{matrix}\right),
\]
where $P^+_L=\Pi^+(\calB)+P_L$. Then the domain is given by
\[
\dom(\calD_{j,\aps})=\left\{s\in\dom(\calD_{j,\max}):\,P_\aps(L)(s|_{\p M_j^{\rm cut}})=0\right\}.
\]

Consider a path connecting the above two boundary conditions with the following domain
\[
\dom(\calD_{j,\theta})=\left\{s\in\dom(\calD_{j,\max}):\,P_\theta(s|_{\p M_j^{\rm cut}})=0\right\},\quad 0\le\theta\le\frac{\pi}{4},
\]
where (as in Kirk--Lesch \cite{KirkLesch04})
\begin{equation}\label{E:bdry proj}
P_\theta:=\left(
\begin{matrix}
P^+_L\cos^2\theta+(1-P^+_L)\sin^2\theta & -\cos\theta\sin\theta \\
-\cos\theta\sin\theta & (1-P^+_L)\cos^2\theta+P^+_L\sin^2\theta
\end{matrix}\right).
\end{equation}
It is clear that $\calD_{j,0}=\calD_{j,\aps}$ and $\calD_{j,\pi/4}=\calD_{j,\Delta}$. 

For $j=0,1$ and $\theta\in[0,\pi/4]$, let
\[
B_{j,\theta}:=\big\{s|_{\p M_j^{\rm cut}}:s\in\dom(\calD_{j,\theta})\big\}
\]
be the space of boundary values. Then $B_{j,\theta}\subset H^{1/2}(E_j|_{\p M_j^{\rm cut}})$ is a self-adjoint elliptic boundary condition in the sense of B\"ar--Ballmann \cite{BaerBallmann12,BaerBallmann16}. It then follows from \cite[Section 8]{BaerBallmann12} that each $\calD_{j,\theta}$ is a self-adjoint Dirac--Schr\"odinger operator, thus is Fredholm. Using the arguments in \cite[Lemma~6.5, Proposition~6.6]{Shi22}, one can show that the relative eta invariant can still be defined with such boundary conditions.

\begin{lemma}\label{L:rel-eta bdry}
Under Assumption~\ref{A:bounded geom}, the relative eta function $\eta(s;\calD_{1,\theta'},\calD_{0,\theta})$ as in \eqref{E:rel-eta func} is well-defined when $\Re(s)$ is large and admits a meromorphic continuation to the whole complex plane for $\theta,\theta'\in[0,\pi/4]$. Therefore, the relative eta invariant $\eta(\calD_{1,\theta'},\calD_{0,\theta})$ exists by Definition~\ref{D:rel-eta}.
\end{lemma}

\begin{remark}\label{R:rel-eta}
The assumptions in Proposition~\ref{P:rel-eta} and Lemma~\ref{L:rel-eta bdry} are mainly used to guarantee that $\calD_1e^{-t\calD_1^2}-\calD_0e^{-t\calD_0^2}$ (or $\calD_{1,\theta'}e^{-t\calD_{1,\theta'}^2}-\calD_{0,\theta}e^{-t\calD_{0,\theta}^2}$) is a trace-class operator. See Bunke \cite{Bunke92,Bunke93comparison}.


Note that the relative trace in \eqref{E:rel-eta func} can be separated as individual compact parts and a relative non-compact part, which can then be written as the integration of pointwise relative trace. In the case that $\calD_0$ (or $\calD_1$, $\calD_{0,\theta}$, $\calD_{1,\theta'}$) is replaced by a conjugation $U_0^*\calD_0U_0$ by a unitary operator $U_0$, if outside a compact subset, $U_0$ is a fiberwise unitary bundle endomorphism and $U_0^*\calD_0U_0=\calD_0$, then the relative eta function is unchanged.

\end{remark}

From the above construction, the operator $\calD_{j,\Delta}=\calD_{j,\pi/4}$ on $M_j^{\rm cut}$ can be just identified with the original operator $\calD_j$ on $M_j$. In this perspective, the gluing formula can be formulated as the following.

\begin{theorem}[Gluing formula]\label{T:gluing}
Let $M_0$ and $M_1$ be two non-compact complete Riemannian manifolds without boundary, $\calD_0$ and $\calD_1$ be two Dirac--Schr\"odinger operators on $(M_0,E_0)$ and $(M_1,E_1)$, respectively, which coincide at infinity. Let $\Sigma\subset M_j$ be a cutting hypersurface chosen as above such that the metrics, bundles and operators are product forms near $\Sigma$ and let $\calD_{j,\theta}$ be the resulting operator on $M_j^{\rm cut}$ with boundary condition determined by $P_\theta$ in \eqref{E:bdry proj}.

Suppose there exists a Lagrangian subspace $L$ of $\ker\calB$. If $\calD_j$ ($j=0,1$) satisfies Assumption~\ref{A:bounded geom} and has an \emph{empty} essential support, then $\eta(\calD_{1,\theta'},\calD_{0,\theta})$ is constant for $\theta,\theta'\in[0,\pi/4]$. In particular,
\[
\eta(\calD_1,\calD_0)=\eta(\calD_{1,\aps},\calD_{0,\aps}),
\]
where the APS boundary condition is associated to $L$.
\end{theorem}

When $\Sigma$ cuts $M_j$ into two parts with one part being compact, the condition that $\ker\calB$ admits a Lagrangian subspace is automatically satisfied as explained in footnote \ref{FN:Lagran}. In this case, one deduces the following additivity for the relative eta invariant.

\begin{corollary}\label{C:splitting}
For $j=0,1$, let $M_j$ be as in Theorem~\ref{T:gluing} and $\calD_j$ be Dirac--Schr\"odinger operators on $M_j$ which coincide at infinity. Let $\Sigma\subset M_j$ be a cutting hypersurface chosen as in the beginning of Subsection~\ref{SS:rel-eta bdry} such that the metrics, bundles and operators are product forms near $\Sigma$.

Suppose that $M_j$ is partitioned by $\Sigma$ into two components, i.e., $M_j^{\rm cut}=M_j'\cup_\Sigma M_j''$, where $M_j'$ is compact and $M_0''\cong M_1''$. If $\calD_j$ ($j=0,1$) satisfies Assumption~\ref{A:bounded geom} and has an \emph{empty} essential support, then
\[
\eta(\calD_1,\calD_0)=\eta(\calD'_{1,\aps})-\eta(\calD'_{0,\aps}),
\]
where $\calD'_{j,\aps}$ denotes the operator $\calD_j$ restricted to $M_j'$ with the APS boundary condition (associated to a Lagrangian subspace $L$ of $\ker\calB$), and the two terms in the right-hand side are usual eta invariants on compact manifolds with boundary.
\end{corollary}

\subsection{Proof of the gluing formula}\label{SS:pf-gluing}

In last subsection, we have constructed a family of self-adjoint Fredholm operators $\calD_{j,\theta}:\dom(\calD_{j,\theta})\to L^2(E_j|_{M_j^{\rm cut}})$. In the following, we will think of $\calD_{j,\theta}$ as operators that act on a fixed domain. This is because there is actually a continuous family of unitary operators $\Omega_{j,\theta}$ on $L^2(E_j|_{M_j^{\rm cut}})$ which map $\dom(\calD_{j,0})$ to $\dom(\calD_{j,\theta})$ (cf. \cite[Section~6.2]{Shi22}). Therefore, one can identify each $\calD_{j,\theta}$ with $\Omega_{j,\theta}^*\calD_{j,\theta}\Omega_{j,\theta}:\dom(\calD_{j,0})\to L^2(E_j|_{M_j^{\rm cut}})$. In fact, each $\Omega_{j,\theta}$ is equal to the identity map outside a compact subset containing a neighborhood of $\Sigma$, so that conjugation by $\Omega_{j,\theta}$ preserves the value of the relative eta invariant (cf. Remark~\ref{R:rel-eta}).

From this discussion, the path $\calD_{j,\theta}$ connecting $\calD_{j,\aps}$ and $\calD_{j,\Delta}$ corresponds to a graph continuous family of self-adjoint Fredholm operators. Thus one can define the \emph{spectral flow} of $\calD_{j,\theta}$, $\theta\in[\underline{\theta},\bar{\theta}]$, denoted by $\spf(\calD_{j,\theta})_{[\underline{\theta},\bar{\theta}]}$, to be the difference of the number of eigenvalues that change from negative to non-negative and the number of eigenvalues that change from non-negative to negative as $\theta$ varies from $\underline{\theta}$ to $\bar{\theta}$ (cf. \cite{BoossLeschPhillips05}).

On the other hand, as shown in \cite[Proposition~6.8]{Shi22}, which generalizes \cite[Theorem~5.8]{Shi22} to manifolds with boundary, for a fixed $\theta_0$, the mod $\ZZ$ reduction of the relative eta invariant, denoted by $\oeta(\calD_{1,\theta},\calD_{0,\theta_0})$, depends smoothly on $\theta$. Moreover, the proof of \cite[Proposition~5.10]{Shi22} can be repeated verbatim in the case of manifolds with boundary, so that the following equation holds.

\begin{lemma}\label{L:rel-eta-spf}
Let the notations be as above. Then for $0\le\underline{\theta}\le\bar{\theta}\le\pi/4$,
\[
\xi(\calD_{1,\bar{\theta}},\calD_{0,\theta_0})-\xi(\calD_{1,\underline{\theta}},\calD_{0,\theta_0}) 
-\frac{1}{2}\int_{\underline{\theta}}^{\bar{\theta}}\left(\frac{d}{d\theta}\oeta(\calD_{1,\theta},\calD_{0,\theta_0})\right)d\theta=\spf(\calD_{1,\theta})_{[\underline{\theta},\bar{\theta}]}.
\]
\end{lemma}

Since it has been shown by Br\"uning--Lesch in \cite{BruningLesch99} that $\frac{d}{d\theta}\oeta(\calD_{1,\theta},\calD_{0,\theta_0})$ vanishes, we conclude that
\begin{equation}\label{E:rel-eta-spf}
\xi(\calD_{1,\bar{\theta}},\calD_{0,\theta_0})-\xi(\calD_{1,\underline{\theta}},\calD_{0,\theta_0})=\spf(\calD_{1,\theta})_{[\underline{\theta},\bar{\theta}]}.
\end{equation}
Similarly,
\begin{equation}\label{E:rel-eta-spf 2}
\xi(\calD_{1,\theta_0},\calD_{0,\bar{\theta}})-\xi(\calD_{1,\theta_0},\calD_{0,\underline{\theta}})=\spf(\calD_{0,\theta})_{[\underline{\theta},\bar{\theta}]}.
\end{equation}
In order to prove Theorem~\ref{T:gluing}, it suffices to show the right-hand sides of \eqref{E:rel-eta-spf} and \eqref{E:rel-eta-spf 2} vanish, which can be deduced by the invertibility of $\calD_{j,\theta}$.

\begin{lemma}\label{L:invertibility}
If $\calD_j$ ($j=0,1$) is a Dirac--Schr\"odinger operator which has an empty essential support, then, for any $\theta\in[0,\pi/4]$, the operator $\calD_{j,\theta}$ is invertible.
\end{lemma}

\begin{proof}
The following computation is the same for $\calD_0$ and $\calD_1$, so we will suppress the subscript ``$j$''.

Choose any $s\in C^\infty(E|_{M^{\rm cut}})$ with compact support such that $P_\theta(s|_{\p M^{\rm cut}})=0$, where $P_\theta$ is the projection in \eqref{E:bdry proj}. Using the Green's formula\footnote{It is of the form
\[
\langle As,s'\rangle_{L^2(E)}=\langle s,A^*s'\rangle_{L^2(E)}-\langle\sigma_A(\nu)s,s'\rangle_{L^2(E|_\pM)},\quad\mbox{for }s,s'\in C^\infty_c(E),
\]
where $\sigma_A$ is the principal symbol of a first-order differential operator $A$, and $\nu$ is the inward-pointing unit normal vector field along $\pM$.}, the Weitzenb\"ock formula \eqref{E:Weitz for} for $\calD$, \eqref{E:decomp of D}, and the fact that $\calD$ has an empty essential support, one gets
\begin{equation}\label{E:calDs square}
\begin{aligned}
\|\calD s\|_{L^2(E_{M^{\rm cut}})}^2& =\langle\calD s,\calD s\rangle_{L^2(E_{M^{\rm cut}})} \\
& =\langle\calD^2 s,s\rangle_{L^2(E_{M^{\rm cut}})}+\langle\tilde{\rmc}(\nu)\calD s,s\rangle_{L^2(E|_{\pM^{\rm cut}})} \\
& \ge\,a\|s\|^2_{L^2(E_{M^{\rm cut}})}+\langle\nabla^*\nabla s,s\rangle_{L^2(E_{M^{\rm cut}})}-\langle\p_us,s\rangle_{L^2(E|_{\pM^{\rm cut}})} \\
& \quad-\langle\widetilde{\calB}s,s\rangle_{L^2(E|_{\pM^{\rm cut}})},
\end{aligned}
\end{equation}
where $a>0$ is the uniform lower bound of $\calD^2-\nabla^*\nabla$.

We first look at the last term in the right-hand side of \eqref{E:calDs square}. Since $P_\theta(s|_{\p M^{\rm cut}})=0$, one can express $s|_{\p M^{\rm cut}}$ in the form
\[
s|_{\p M^{\rm cut}}=\left(
\begin{matrix}
f^+\sin\theta+f^-\cos\theta \\
f^-\sin\theta+f^+\cos\theta
\end{matrix}
\right),\qquad\text{with }f^-\in\ker P_L^+,\ f^+\in\ker(1-P_L^+).
\]
It follows that $\langle\calB f^+,f^+\rangle_{L^2(E|_\Sigma)}\ge0$, $\langle\calB f^-,f^-\rangle_{L^2(E|_\Sigma)}\le0$. Hence
\[
\begin{aligned}
\langle\widetilde{\calB}s,s\rangle_{L^2(E|_{\pM^{\rm cut}})}&=\big\langle\calB(f^+\sin\theta+f^-\cos\theta),f^+\sin\theta+f^-\cos\theta\big\rangle_{L^2(E|_\Sigma)} \\
&\ \ \ -\big\langle\calB(f^-\sin\theta+f^+\cos\theta),f^-\sin\theta+f^+\cos\theta\big\rangle_{L^2(E|_\Sigma)} \\
&=(\sin^2\theta-\cos^2\theta)\big[\langle\calB f^+,f^+\rangle_{L^2(E|_\Sigma)}-\langle\calB f^-,f^-\rangle_{L^2(E|_\Sigma)}\big]\le0.
\end{aligned}
\]

For the other terms in \eqref{E:calDs square}, apply the Green's formula again,
\[
\begin{aligned}
&\langle\nabla^*\nabla s,s\rangle_{L^2(E_{M^{\rm cut}})}-\|\nabla s\|^2_{L^2(E_{M^{\rm cut}})} \\
&=-\langle\sigma_{\nabla^*}(\nu)\nabla s,s\rangle_{L^2(E|_{\pM^{\rm cut}})}
=\langle\nabla s,\sigma_\nabla(\nu)s\rangle_{L^2(E|_{\pM^{\rm cut}})} \\
&=\langle\nabla s,\nu\otimes s\rangle_{L^2(E|_{\pM^{\rm cut}})}
=\langle\nabla_\nu s,s\rangle_{L^2(E|_{\pM^{\rm cut}})} \\
&=\langle\p_us,s\rangle_{L^2(E|_{\pM^{\rm cut}})}.
\end{aligned}
\]
So
\[
\|\calD s\|_{L^2(E_{M^{\rm cut}})}^2\ge a\|s\|^2_{L^2(E_{M^{\rm cut}})}+\|\nabla s\|^2_{L^2(E_{M^{\rm cut}})}\ge a\|s\|^2_{L^2(E_{M^{\rm cut}})}.
\]
Therefore, $\calD_{j,\theta}$ is invertible and $\ker\calD_{j,\theta}=\{0\}$.
\end{proof}

\begin{proof}[Proof of Theorem~\ref{T:gluing}]
By Lemma~\ref{L:invertibility}, $\spf(\calD_{j,\theta})$ vanishes over each subinterval of $[0,\pi/4]$. Theorem~\ref{T:gluing} is now an immediate consequence of \eqref{E:rel-eta-spf} and \eqref{E:rel-eta-spf 2}.
\end{proof}

\section{The spectral flow on non-compact manifolds}\label{S:sp flow}

In this section, we use relative eta invariants to derive a formula calculating the spectral flow of a family of Dirac--Schr\"odinger operators on non-compact manifolds without boundary. This generalizes a result of Getzler \cite{Getzler93} and can be used to study positive scalar curvature problems.

\subsection{Variation of the truncated relative eta invariants}\label{SS:var trun rel-eta}

Let $M$ be a non-compact complete Riemannian $n$-manifold without boundary endowed with a Dirac bundle $E$, and $\calD_j$ ($j=0,1$) be two Dirac--Schr\"odinger operators on $(M,E)$ which coincide at infinity. Assume $M$ has bounded sectional curvature. For $\varepsilon>0$, consider
\[
\begin{aligned}
\eta^\varepsilon(s;\calD_1,\calD_0)& :=\frac{1}{\Gamma((s+1)/2)}\int_0^\varepsilon t^{(s-1)/2}\Tr\big(\calD_1e^{-t\calD_1^2}-\calD_0e^{-t\calD_0^2}\big)dt, \\
\eta_\varepsilon(s;\calD_1,\calD_0)& :=\frac{1}{\Gamma((s+1)/2)}\int_\varepsilon^\infty t^{(s-1)/2}\Tr\big(\calD_1e^{-t\calD_1^2}-\calD_0e^{-t\calD_0^2}\big)dt.
\end{aligned}
\]
Then by \cite[Section~4]{Shi22}, the first integral is absolutely convergent, is holomorphic for $\Re(s)>n$, and admits a meromorphic continuation to the whole complex plane which is regular at $s=0$; while the second integral is absolutely convergent for $s$ in the whole complex plane. Thus the \emph{truncated} relative eta invariants
\[
\begin{aligned}
\eta^\varepsilon(\calD_1,\calD_0)& :=\eta^\varepsilon(0;\calD_1,\calD_0), \\
\eta_\varepsilon(\calD_1,\calD_0)& :=\eta_\varepsilon(0;\calD_1,\calD_0) \\
& \;=\frac{1}{\pi^{1/2}}\int_\varepsilon^\infty t^{-1/2}\Tr\big(\calD_1e^{-t\calD_1^2}-\calD_0e^{-t\calD_0^2}\big)dt
\end{aligned}
\]
are well-defined.

Let $\calD_{1,r}$, $r\in[0,1]$ be a smooth family of Dirac--Schr\"odinger operators on $(M,E)$ which coincide with $\calD_1$ at infinity. As in \cite[Section~5.2]{Shi22}, we have the following variation formula for $\eta^\varepsilon(\calD_{1,r},\calD_0)$.

\begin{lemma}\label{L:var trun rel-eta}
Suppose there exists the following asymptotic expansion
\begin{equation}\label{E:STexpan-calD}
\Tr\big(\dot{\calD}_{1,r}e^{-t\calD_{1,r}^2}\big)\,\sim\,\sum_{k=0}^\infty c_k(r)t^{(k-n-1)/2},\qquad\text{as }t\to0.
\end{equation}
Then
\[
\frac{d}{dr}\eta^\varepsilon(\calD_{1,r},\calD_0)=2\Big(\frac{\varepsilon}{\pi}\Big)^{1/2}\Tr(\dot{\calD}_{1,r}e^{-\varepsilon\calD_{1,r}^2})-\frac{2}{\pi^{1/2}}c_n(r),
\]
where $\dot{\calD}_{1,r}=\frac{d}{dr}\calD_{1,r}$.
\end{lemma}

\begin{proof}
Note that
\[
\frac{\p}{\p r}\Tr\big(\calD_{1,r}e^{-t\calD_{1,r}^2}-\calD_0e^{-t\calD_0^2}\big)=\Big(1+2t\frac{\p}{\p t}\Big)\Tr(\dot{\calD}_{1,r}e^{-t\calD_{1,r}^2}).
\]
Here $\dot{\calD}_{1,r}e^{-t\calD_{1,r}^2}$ is a trace-class operator by \cite[Lemma~5.6]{Shi22}, as $\dot{\calD}_{1,r}$ vanishes at infinity. So for $\Re(s)>n$, by \eqref{E:STexpan-calD},
\[
\begin{aligned}
&\frac{d}{dr}\eta^\varepsilon(s;\calD_{1,r},\calD_0) \\
&\quad=\frac{1}{\Gamma((s+1)/2)}\left(2\varepsilon^{(s+1)/2}\Tr(\dot{\calD}_{1,r}e^{-\varepsilon\calD_{1,r}^2})-s\int_0^\varepsilon t^{(s+1)/2}\Tr(\dot{\calD}_{1,r}e^{-t\calD_{1,r}^2})\right).
\end{aligned}
\]
Again by \eqref{E:STexpan-calD}, the integral in the right-hand side admits a meromorphic continuation to the complex plane such that $s=0$ is a simple pole with residue $2c_n(r)$. The lemma then follows.
\end{proof}

\subsection{The spectral flow formula}\label{SS:sp-flow formula}

We now consider $\calD_{1,r}$, $r\in[0,1]$ to be a family of Dirac--Schr\"odinger operators connecting $\calD_0$ and $\calD_1$, i.e., $\calD_{1,0}=\calD_0,$ $\calD_{1,1}=\calD_1$, and each $\calD_r$ coincides with $\calD_0$ at infinity. To simplify notations, we will denote $\calD_{1,r}$ by $\calD_r$ in the following. By the reason as in Subsection~\ref{SS:pf-gluing}, the spectral flow $\spf(\calD_r)_{[0,1]}$ is well-defined. In \cite[Proposition~5.10]{Shi22} (see Lemma~\ref{L:rel-eta-spf}), a formula calculating the spectral flow using relative eta invariant has already been obtained. Now we derive another formula in terms of truncated relative eta invariant, which has the form of Getzler \cite[Theorem~2.6]{Getzler93}.

\begin{proposition}\label{P:sp-flow trun rel-eta}
Let $M$ be a non-compact complete Riemannian manifold with bounded sectional curvature, and $\calD_0$ and $\calD_1$ be two Dirac--Schr\"odinger operators on $M$ which coincide at infinity. Suppose $\calD_r$, $r\in[0,1]$ is a smooth family of Dirac--Schr\"odinger operators connecting $\calD_0$ and $\calD_1$ such that each $\calD_r$ coincides with $\calD_0$ at infinity. Then for $\varepsilon>0$,
\begin{equation}\label{E:sp-flow trun rel-eta}
\begin{aligned}
\spf(\calD_r)_{[0,1]}& =\Big(\frac{\varepsilon}{\pi}\Big)^{1/2}\int_0^1\Tr(\dot{\calD}_re^{-\varepsilon\calD_r^2})dr \\
& \quad+\frac{1}{2}\big(\eta_\varepsilon(\calD_1,\calD_0)+\dim\ker\calD_1-\dim\ker\calD_0\big).
\end{aligned}
\end{equation}
In particular, when $\dim\ker\calD_0=\dim\ker\calD_1$, we have
\[
\spf(\calD_r)_{[0,1]}=\Big(\frac{\varepsilon}{\pi}\Big)^{1/2}\int_0^1\Tr(\dot{\calD}_re^{-\varepsilon\calD_r^2})dr+\frac{1}{2}\eta_\varepsilon(\calD_1,\calD_0).
\]
\end{proposition}

\begin{proof}
Since $\dot{\calD_r}$ is compactly supported, the asymptotic expansion \eqref{E:STexpan-calD} exists. Hence by Lemma~\ref{L:var trun rel-eta},
\[
\begin{aligned}
\eta^\varepsilon(\calD_1,\calD_0)&=\eta^\varepsilon(\calD_0,\calD_0)+\int_0^1\left(\frac{d}{dr}\eta^\varepsilon(\calD_r,\calD_0)\right)dr \\
&=2\int_0^1\left(\Big(\frac{\varepsilon}{\pi}\Big)^{1/2}\Tr(\dot{\calD}_{1,r}e^{-\varepsilon\calD_{1,r}^2})-\frac{1}{\pi^{1/2}}c_n(r)\right)dr.
\end{aligned}
\]
On the other hand, by \cite[Theorem~5.8, Proposition~5.10]{Shi22},
\[
\spf(\calD_r)_{[0,1]}=\frac{1}{2}\big(\eta(\calD_1,\calD_0)+\dim\ker\calD_1-\dim\ker\calD_0\big)+\int_0^1\frac{1}{\pi^{1/2}}c_n(r)dr.
\]
Combining the above equations yields \eqref{E:sp-flow trun rel-eta}.
\end{proof}

\subsection{Chern--Simons forms and the spectral flow}\label{SS:Chern-Simons sp-flow }

In this subsection, we focus on spin manifolds, and the above result will transform to a geometric formula.

As in Example~\ref{Eg:Dirac-Schro-1}, let $(M,g)$ be an odd-dimensional Riemannian spin manifold without boundary admitting a complete metric of uniformly PSC outside a compact subset, and $\slaS\to M$ be the spinor bundle. Suppose $F\to M$ is a Hermitian vector bundle with two flat connections $\nabla_0^F$ and $\nabla_1^F$ which coincide at infinity. For $r\in[0,1]$, put $\nabla_r^F=(1-r)\nabla_0^F+r\nabla_1^F$, which induces a family of connections
\[
\nabla^E_r=\nabla^\slaS\otimes1+1\otimes\nabla_r^F,\quad 0\le r\le1
\]
on the twisted spinor bundle $E=\slaS\otimes F$. So we obtain a family of Dirac--Schr\"odinger operators $\slaD_{E,r}$, $r\in[0,1]$.

Recall the Chern character form associated to a connection $\nabla$ of a vector bundle is the even-degree differential form defined by
\[
\ch(\nabla):=\tr\left(\exp\Big(\frac{\sqrt{-1}}{2\pi}\nabla^2\Big)\right).
\]
For two connections $\nabla_0,\nabla_1$ on a vector bundle, their \emph{Chern--Simons transgressed form} associated to the Chern character is the odd-degree differential form
\[
\Tch(\nabla_0,\nabla_1)=-\int_0^1\tr\left(\frac{\sqrt{-1}}{2\pi}\dot{\nabla}_r\exp\Big(\frac{\sqrt{-1}}{2\pi}\nabla_r^2\Big)\right)dr,
\]
where $\nabla_r=(1-r)\nabla_0+r\nabla_1$ and $\dot{\nabla}_r=\nabla_1-\nabla_0$. It satisfies the transgression formula (cf. Zhang \cite[Chapter~1]{Zhang01book})
\[
\ch(\nabla_0)-\ch(\nabla_1)=d\Tch(\nabla_0,\nabla_1).
\]
If $\nabla_0$ and $\nabla_1$ are both flat connections, then $\Tch(\nabla_0,\nabla_1)$ is a closed form. Let $\omega=\nabla_1-\nabla_0$. Then it can be derived like in Getzler \cite[Section~1]{Getzler93} that
\begin{equation}\label{E:Chern-Simons}
\Tch(\nabla_0,\nabla_1)=\sum_{k=0}^\infty\Big(\frac{1}{2\pi\sqrt{-1}}\Big)^{k+1}\frac{k!}{(2k+1)!}\tr(\omega^{2k+1}).
\end{equation}

With the above notations, the spectral flow of the path $\slaD_{E,r}$, $r\in[0,1]$ can be computed as the following.

\begin{theorem}\label{T:sp flow-trans}
Suppose $(M,g)$ is a $(2m+1)$-dimensional non-compact complete Riemannian spin manifold of bounded sectional curvature such that the scalar curvature is uniformly positive at infinity. Let $\nabla_0^F$ and $\nabla_1^F$ be two flat connections on a Hermitian vector bundle $F$ over $M$ which coincide at infinity, and $\nabla_r^F$, $r\in[0,1]$ be the linear path between them. Let $\slaD_{E,r}$ be the associated family of twisted spin Dirac operators on $E=\slaS\otimes F$ as above. Then
\begin{equation}\label{E:sp flow-trans}
\begin{aligned}
\spf(\slaD_{E,r})_{[0,1]}& =\int_M\hat{A}(M,g)\Tch(\nabla_0^F,\nabla_1^F) \\
& \quad+\frac{1}{2}\big(\eta(\slaD_{E,1},\slaD_{E,0})+\dim\ker\slaD_{E,1}-\dim\ker\slaD_{E,0}\big),
\end{aligned}
\end{equation}
where $\Tch(\nabla_0^F,\nabla_1^F)$ is the Chern--Simons form given in \eqref{E:Chern-Simons}, and
\begin{equation}\label{E:A-hat form}
\hat{A}(M,g):=\det{}^{1/2}\left(\frac{\frac{\sqrt{-1}}{4\pi}R^{TM}}{\sinh(\frac{\sqrt{-1}}{4\pi}R^{TM})}\right)
\end{equation}
is the $\hat{A}$-genus form of $(M,g)$ (with $R^{TM}$ denoting the Riemannian curvature associated to the Levi--Civita connection of $g$).
\end{theorem}

\begin{remark}\label{R:sp flow-spin}
By the hypothesis that $\nabla_0^F$ and $\nabla_1^F$ coincide at infinity, one has that $\Tch(\nabla_0^F,\nabla_1^F)$ is compactly supported. Thus the integral in \eqref{E:sp flow-trans} is well-defined.
\end{remark}

\begin{remark}\label{R:sp flow-spin-1}
Formula \eqref{E:sp flow-trans} implies the following mod $\ZZ$ formula for the reduced relative eta invariant (compare Gilkey \cite[Theorem~3.11.6]{Gilkey95book})
\[
\xi(\slaD_{E,1},\slaD_{E,0})=-\int_M\hat{A}(M,g)\Tch(\nabla_0^F,\nabla_1^F)\quad\mbox{mod }\ZZ.
\]
\end{remark}

\begin{proof}[Proof of Theorem~\ref{T:sp flow-trans}]
In view of Proposition~\ref{P:sp-flow trun rel-eta}, it suffices to prove that
\[
\lim_{\varepsilon\to0}\Big(\frac{\varepsilon}{\pi}\Big)^{1/2}\int_0^1\Tr(\dot{\slaD}_{E,r}e^{-\varepsilon\slaD_{E,r}^2})dr=\int_M\hat{A}(M,g)\Tch(\nabla_0^F,\nabla_1^F).
\]

Cut $M$  along a compact hypersurface such that $M$ is divided into two parts $M'$ and $M''$, where $M'$ is a compact manifold with boundary containing the support of $\nabla_1^F-\nabla_0^F$. Set $\widehat{M}$ to be the closed double of $M'$. Then $g$ and $\slaD_{E,r}$ can be extended to $\widehat{M}$, which are denoted by $\hat{g}$ and $\widehat{\slaD}_{E,r}$, respectively.
By \cite[Section~3]{Shi22}, when $\varepsilon\to0$, one can replace $\Tr(\dot{\slaD}_{E,r}e^{-\varepsilon\slaD_{E,r}^2})$ by $\Tr(\dot{\widehat{\slaD}}_{E,r}e^{-\varepsilon\widehat{\slaD}_{E,r}^2})$. Since the latter is a heat trace on a closed manifold, by the local index computation (cf. Getzler \cite[pp. 499--500]{Getzler93}), we get\footnote{Note that our convention in the definition of $\hat{A}$-genus form and Chern--Simons form includes the factor $\frac{\sqrt{-1}}{2\pi}$, which is different from that in \cite{Getzler93}.}
\[
\begin{aligned}
&\lim_{\varepsilon\to0}\Big(\frac{\varepsilon}{\pi}\Big)^{1/2}\int_0^1\Tr(\dot{\widehat{\slaD}}_{E,r}e^{-\varepsilon\widehat{\slaD}_{E,r}^2})dr \\
&=\frac{(-\sqrt{-1})^{2m+1}(2\sqrt{-1})^m}{\pi^{1/2}(4\pi)^{m+1/2}}\int_0^1\int_{\widehat{M}}\det{}^{1/2}\left(\frac{R^{T\widehat{M}}/2}{\sinh R^{T\widehat{M}}/2}\right)\wedge\tr\big(\widehat{\omega}^F\exp(-(\widehat{\nabla}_r^F)^2)\big)dr \\
&=\frac{1}{(2\pi\sqrt{-1})^{m+1}}\int_{\widehat{M}}\det{}^{1/2}\left(\frac{R^{T\widehat{M}}/2}{\sinh R^{T\widehat{M}}/2}\right)\wedge\int_0^1\tr\big(\widehat{\omega}^F\exp(-(\widehat{\nabla}_r^F)^2)\big)dr \\
&=\int_{\widehat{M}}\hat{A}(\widehat{M},\hat{g})\Tch(\widehat{\nabla}_0^F,\widehat{\nabla}_1^F),
\end{aligned}
\]
where $\widehat{\nabla}_j^F$ ($j=0,1$) is the extension of $\nabla_j^F$ to $\widehat{M}$ such that $\widehat{\omega}^F=\widehat{\nabla}_1^F-\widehat{\nabla}_0^F$ vanishes outside $M'$. It is clear that $\hat{A}(M,g)\Tch(\nabla_0^F,\nabla_1^F)$ and $\hat{A}(\widehat{M},\hat{g})\Tch(\widehat{\nabla}_0^F,\widehat{\nabla}_1^F)$ are equal on their common support. Therefore, \eqref{E:sp flow-trans} is proved.
\end{proof}

The above setting has an important special case as follows. Let $N$ be a closed manifold and $u:N\to U_l(\CC)$ be a smooth map with values in a unitary group for some integer $l>0$. Then $u$ induces a family of connections
\[\label{E:conn family}
\nabla_r:=d+ru^{-1}(du),\quad 0\le r\le1
\]
on the trivial bundle $N\times\CC^l$. In this setting, the Chern--Simons form \eqref{E:Chern-Simons} is just the \emph{odd Chern character form} (cf. Zhang \cite[Chapter~1]{Zhang01book})
\[\label{E:odd chern ch}
\ch(u):=\sum_{k=0}^\infty\Big(\frac{1}{2\pi\sqrt{-1}}\Big)^{k+1}\frac{k!}{(2k+1)!}\tr\big((u^{-1}(du))^{2k+1}\big).
\]

Suppose $N$ is of the same dimension as $M$. Let $f:M\to N$ be a smooth map of non-zero degree which is constant outside a compact subset. On the pull-back bundle $F:=f^*(N\times\CC^l)$ over $M$, one has a family of connections $\nabla^F_r=f^*(\nabla_r)$ for $r\in[0,1]$. Like before, this induces a smooth family of twisted spin Dirac operators $\slaD_{E,r}$, $r\in[0,1]$ on the twisted bundle $E=\slaS\otimes F$.

Note that $\slaD_{E,1}=(u\circ f)^{-1}\slaD_{E,0}(u\circ f)$. Thus $\slaD_{E,0}$ and $\slaD_{E,1}$ are conjugated by a fiberwise unitary operator and they coincide at infinity, which means
\[
\eta(\slaD_{E,1},\slaD_{E,0})=0=\dim\ker\slaD_{E,1}-\dim\ker\slaD_{E,0}.
\]
Therefore we obtain the following consequence of Theorem~\ref{T:sp flow-trans}.

\begin{corollary}\label{C:sp flow-pull back}
Let $\slaD_{E,r}$, $r\in[0,1]$ be the family of operators on $(M,g,E)$ defined as above. Then
\[
\spf(\slaD_{E,r})_{[0,1]}=\int_M\hat{A}(M,g)f^*\ch(u).
\]
\end{corollary}

\subsection{Area enlargeable manifolds}\label{SS:area enlar}

For a long time, the index theory of Dirac operators has been applied in the study of PSC problems on spin manifolds. Traditionally, it is carried out in even dimensions. And when the dimension is odd, one usually takes product with $S^1$ to convert to even-dimensional case.
In \cite{LiSuWang24}, Li, Su and Wang use the method of spectral flow to give a direct proof of Llarull's theorem \cite{Llarull98} (which is a question asked by Gromov \cite{Gromov23Four}) and its generalization in odd dimensions.
Inspired by it, in this subsection, we will apply the formula for the spectral flow on non-compact manifolds obtained above to provide a new proof of Gromov--Lawson's theorem about area enlargeable manifolds in odd dimensions.

We first recall the notion of area enlargeable manifolds.

\begin{definition}[Gromov--Lawson \cite{GromovLawson83}]\label{D:area enlar}
A connected $n$-manifold $M$ (without boundary) is called \emph{area enlargeable} (or \emph{$\Lambda^2$-enlargeable}) if given any Riemannian metric on $M$ and any  $\varepsilon>0$, there exist a covering manifold $\widetilde{M}\to M$ (with the lifted metric) which is spin, and a smooth map $f:\widetilde{M}\to S^n(1)$ (the standard unit sphere) which is constant at infinity and has non-zero degree such that $f$ is $\varepsilon$\emph{-contracting} on two forms, which means $|f^*\alpha|\le\varepsilon|\alpha|$ for all 2-forms $\alpha$ on $S^n(1)$. Here $|\cdot|$ denotes the pointwise norm induced by the inner product given by the Riemannian metric.
\end{definition}

There is a stronger version of Definition~\ref{D:area enlar} where the 2-forms are replaced by 1-forms. In this case, the manifold is called \emph{enlargeable}. This was first introduced by Gromov--Lawson in \cite{GromovLawson80spin}. Intuitively, the ``largeness'' of an (area) enlargeable manifold obstructs the existence of a PSC metric. The following is one of the results in this regard.

\begin{theorem}[Gromov--Lawson \cite{GromovLawson83}]\label{T:area enlar}
A compact area enlargeable manifold does not admit a metric of positive scalar curvature.
\end{theorem}

As mentioned in the beginning of this subsection, Theorem~\ref{T:area enlar} was proved by applying a relative index theorem in even dimensions. The following is an alternate proof in odd-dimensional case.

\begin{proof}[Proof of Theorem~\ref{T:area enlar} in odd dimensions]
Suppose $M$ is a compact area enlargeable manifold of dimension $2m-1$ ($m\ge2$) and carries a metric $g$ of scalar curvature $\kappa>0$. Then there exists a constant $\kappa_0>0$ such that $\kappa\ge\kappa_0$ on $M$.

As in Getzler \cite{Getzler93} and Li--Su--Wang \cite{LiSuWang24}, consider the trivial bundle $S^{2m-1}(1)\times\slaS_{2m}^+$ over $S^{2m-1}(1)$, where $\slaS_{2m}^+$ is the Hermitian space of half spinors for $\CC l(\RR^{2m})$, the complexified Clifford algebra of $\RR^{2m}$. Let $\bar{\rmc}(\cdot)$ denote the Clifford multiplication of $\CC l(\RR^{2m})$ on $\slaS_{2m}^+$, and $\bar{\rmc}_i$ $(i=1,\cdots,2m)$ denote $\bar{\rmc}(\p_i)$, where $\{\p_i\}_{i=1}^{2m}$ is the canonical oriented orthonormal basis of $\RR^{2m}$. Introduce $u:S^{2m-1}(1)\to U_{2^{m-1}}(\CC)$ as
\[
u(x)=\bar{\rmc}_{2m}\bar{\rmc}(x)=\sum_{i=1}^{2m}x^i\cdot\bar{\rmc}_{2m}\bar{\rmc}_i.
\]

Let $\widetilde{M}$ (with the lifted metric $\tilg$) and $f:\widetilde{M}\to S^{2m-1}(1)$ be as in Definition~\ref{D:area enlar} for a given $\varepsilon>0$. By the discussion in last subsection, $u$ induces a family of connections on $F=f^*(S^{2m-1}(1)\times\slaS_{2m}^+)$. And we get a family of Dirac--Schr\"odinger operators $\slaD_{E,r}$, $r\in[0,1]$ on the twisted bundle $E=\slaS_{\widetilde{M}}\otimes F$ over $\widetilde{M}$. By Corollary~\ref{C:sp flow-pull back} (see also \cite[Proposition 3.3]{LiSuWang24}),
\begin{equation}\label{E:sp flow-1}
\spf(\slaD_{E,r})_{[0,1]}=\int_{\widetilde{M}}\hat{A}(\widetilde{M},\tilg)f^*\ch(u)=\deg(f)\int_{S^{2m-1}(1)}\ch(u)=\deg(f),
\end{equation}
where the second equality follows from the fact that all the terms in $\ch(u)$ are exact forms except the top-degree part, and the last equality uses a formula of Getzler \cite[Proposition~1.4]{Getzler93}.

On the other hand, we have the Lichnerowicz-type formula
\[
\slaD_{E,r}^2=(\nabla_r^E)^*\nabla_r^E+\frac{\kappa_{\widetilde{M}}}{4}+\calR_r^F,
\]
where $\calR_r^F=\sum_{i<j}\rmc(e_i)\rmc(e_j)R_r^F(e_i,e_j)$ with respect to a local orthonormal frame $\{e_i\}_{i=1}^{2m-1}$ of $T\widetilde{M}$.
When $u$ is chosen as above, $R_r^F$ is explicitly computed in \cite{LiSuWang24}. In particular, if $f$ is $\varepsilon$-contracting on 2-forms, then for any $r\in[0,1]$,
\[
\langle\calR_r^Fs,s\rangle_{L^2(E)}\ge-\frac{(2m-1)(2m-2)}{4}\varepsilon\|s\|_{L^2(E)}^2,\quad\forall s\in L^2(E).
\]
From this and the fact that $\kappa_{\widetilde{M}}\ge\kappa_0>0$, one immediate sees that by choosing $\varepsilon$ small enough, $\slaD_{E,r}$ is invertible for any $r\in[0,1]$. Hence, $\spf(\slaD_{E,r})_{[0,1]}=0$. Since $f$ has non-zero degree, this contradicts \eqref{E:sp flow-1}. Therefore, $M$ cannot admit a PSC metric.
\end{proof}

\section{Index theorem related to manifolds with uniformly PSC metrics}\label{S:ind psc}

In this section, we derive an index formula involving relative eta invariants in the case of uniformly PSC metrics. These results will be applied to the study of uniformly PSC metrics on certain non-compact manifolds in the next section.

\subsection{Space of uniformly PSC metrics which coincide at infinity}\label{SS:space psc}

Let $M$ be a non-compact manifold without boundary, we will consider complete metrics of uniformly positive scalar curvature on $M$ which coincide at infinity. Let $g$ be such a metric and let $\frakR^+_\infty(M,g)$ denote the space of complete uniformly PSC metrics on $M$ which coincide with $g$ outside a compact subset. Endowing $\frakR^+_\infty(M,g)$ with the $C^\infty_{\rm loc}$-topology, we introduce the following equivalence relations on this space, which have already appeared in compact situations.

\begin{definition}\label{D:equiv rela psc}
(1) Two metrics $g_0,g_1\in\frakR^+_\infty(M,g)$ are called \emph{PSC-isotopic} if they lie in the same path component in $\frakR^+_\infty(M,g)$. In this case, a (smooth) path connecting $g_0$ and $g_1$ is called a \emph{PSC-isotopy} between $g_0$ and $g_1$.

(2) Two metrics $g_0,g_1\in\frakR^+_\infty(M,g)$ are called \emph{PSC-concordant} if there exists a smooth metric $g_{0,1}$ of uniformly positive scalar curvature on $M\times[0,a]$ for some $a>0$, such that
\begin{enumerate}
\item $g_{0,1}$ is a product metric near the boundary;
\item $g_{0,1}|_{M\times\{0\}}=g_0$, $g_{0,1}|_{M\times\{a\}}=g_1$;
\item $g_{0,1}$ is a product metric outside a compact subset of $M\times[0,a]$.
\end{enumerate}
In this case, $(M\times[0,a],g_{0,1})$ is called a \emph{PSC-concordance} between $g_0$ and $g_1$.

(3) Two metrics $g_0,g_1\in\frakR^+_\infty(M,g)$ are called \emph{PSC-cobordant} if there exist a manifold $W$ with boundary and a smooth metric $g_W$ of uniformly positive scalar curvature on $W$, such that
\begin{enumerate}
\item $\p W=M\sqcup-M$, where $-M$ denotes $M$ with opposite orientation;
\item $g_W$ is a product metric near the boundary;
\item $g_W|_M=g_0$, $g_W|_{-M}=g_1$;
\item $W$ is isometric to $M'\times[0,b]$ (with product metric) outside a compact subset, where $M'$ is $M$ removing a compact subset.
\end{enumerate}
In this case, $(W,g_W)$ is called a \emph{PSC-cobordism} between $g_0$ and $g_1$.
\end{definition}

\begin{remark}\label{R:equiv rela psc-1}
In cases (2) and (3), if $g_{0,1}$ (resp. $g_W$) is not required to be of positive scalar curvature (which can only happen on an interior relatively compact subset), then $(M\times[0,a],g_{0,1})$ (resp. $(W,g_W$)) is just called a \emph{concordance} (resp. \emph{cobordism}) between $g_0$ and $g_1$.
\end{remark}

\begin{remark}\label{R:equiv rela psc-2}
Like the compact case, it can be shown that PSC-isotopic metrics must be PSC-concordant (cf. Gromov--Lawson \cite[Lemma~3]{GromovLawson80classification}, Rosenberg--Stolz \cite[Proposition~3.3]{RosenbergStolz01}). More clearly, PSC-concordant metrics must be PSC-cobordant. In other words,
\[
\text{PSC-isotopy}\Longrightarrow\text{PSC-concordance}\Longrightarrow\text{PSC-cobordism}.
\]
\end{remark}

\subsection{Index formula on a cobordism for uniformly PSC metrics}\label{SS:ind thm psc}

On a Riemannian spin manifold $(M,g)$, as in Example~\ref{Eg:Dirac-Schro-1}, one can consider the spin Dirac operator $\slaD_\slaS:\dom\slaD_\slaS\to L^2(\slaS)$ and the twisted spin Dirac operator $\slaD_{\slaS\otimes  F}:\dom\slaD_{\slaS\otimes  F}\to L^2(\slaS\otimes F)$, where $\slaS\to M$ is the spinor bundle and $F\to M$ is a Hermitian vector bundle with Hermitian connection. 
If $g$ is a metric of uniformly PSC and \emph{bounded sectional curvature}, for $g_0,g_1\in\frakR^+_\infty(M,g)$, let $\slaS_0$ and $\slaS_1$ be the associated spinor bundles. Then $\slaD_{\slaS_0}$ and $\slaD_{\slaS_1}$ are Dirac--Schr\"odinger operators (with empty essential support). So the relative eta invariant $\eta(\slaD_{\slaS_1},\slaD_{\slaS_0})$ is well-defined (cf. Proposition~\ref{P:rel-eta}). Similarly, the relative eta invariant $\eta(\slaD_{\slaS_1\otimes F},\slaD_{\slaS_0\otimes F})$ of the twisted spin Dirac operators can be defined, provided that the term $\frac{\kappa}{4}+\calR^F$ in \eqref{E:Lich for-1} is uniformly positive outside a compact subset.

\begin{remark}\label{R:rel eta vanish}
In certain dimensions, the relative eta invariants would vanish just like the eta invariants (cf. Atiyah--Patodi--Singer \cite[p. 61]{APS1}). When $\dim M$ is even, the Clifford multiplication of the volume form anti-commutes with $\slaD_{\slaS_j\otimes F}$ ($j=0,1$), which implies the vanishing of $\eta(\slaD_{\slaS_1\otimes F},\slaD_{\slaS_0\otimes F})$. When $\dim M\equiv1$ mod 4, there also exists an involution on $\slaS_j$ which anti-commutes with the untwisted spin Dirac operator $\slaD_{\slaS_j}$. In this case, $\eta(\slaD_{\slaS_1},\slaD_{\slaS_0})$ vanishes.
\end{remark}

Assume $\dim M$ is odd. For $g_0,g_1\in\frakR^+_\infty(M,g)$, let $(W,g_W)$ be a cobordism (not necessarily a PSC-cobordism) between $g_0$ and $g_1$. Suppose the spin structure on $M$ extends over $W$. Then we can talk about the spin Dirac operator $\slaD_{\slaS_W}$ on $W$, where $\slaS_W\to W$ is the corresponding spinor bundle. There is a $\ZZ_2$-grading $\slaS_W=\slaS_W^+\oplus\slaS_W^-$ on the even-dimensional manifold $W$ so that $\slaD_{\slaS_W}=\slaD_{\slaS_W}^+\oplus\slaD_{\slaS_W}^-$ is $\ZZ_2$-graded. Recall that $\p W=M\sqcup-M$. It follows that $\slaD_{\slaS_0}$ and $-\slaD_{\slaS_1}$ are the restrictions of $\slaD_{\slaS_W}^+$ to the two boundary components (with respect to the inward-pointing normal). We elongate $W$ by attaching two half-cylinders $M\times(-\infty,0]$ and $M\times[0,\infty)$ to the corresponding boundary components of $W$ to get a complete manifold without boundary
\[
\widetilde{W}=M\times(-\infty,0]\cup_M W\cup_{-M}M\times[0,\infty).
\]
All the structures can be extended to the cylinder parts in product form. In particular, the resulting metric $\tilde{g}$ on $\widetilde{W}$ will be of uniformly PSC outside a compact subset. Let 
\[
\slaD_{\slaS_{\widetilde{W}}}:\dom\slaD_{\slaS_{\widetilde{W}}}\to L^2(\slaS_{\widetilde{W}})
\]
be the extension of $\slaD_{\slaS_W}$ to $\widetilde{W}$, where again, $\dom\slaD_{\slaS_{\widetilde{W}}}$ consists of all $s\in L^2(\slaS_{\widetilde{W}})$ such that $\slaD_{\slaS_{\widetilde{W}}}s\in L^2(\slaS_{\widetilde{W}})$. Then $\slaD_{\slaS_{\widetilde{W}}}$ is invertible at infinity, thus is a Fredholm operator. Therefore, we can consider the $L^2$-index
\[
\ind\slaD_{\slaS_{\widetilde{W}}}^+:=\dim\ker\slaD_{\slaS_{\widetilde{W}}}^+-\dim\ker\slaD_{\slaS_{\widetilde{W}}}^-.
\]

This index represents the APS index as mentioned in the Introduction and can somehow be thought of as a non-compact generalization of the quantity $i(g_0,g_1)$ considered by Gromov--Lawson \cite[(3.13)]{GromovLawson83}. Theorem~\ref{T:cobor rel-eta}, which can be regarded as an APS-type index formula in the non-compact boundary situation, shows that this index can be computed via the relative eta invariant.

\begin{theorem}\label{T:cobor rel-eta}
As described above, let $(W,g_W)$ be a cobordism between two metrics $g_0,g_1\in\frakR_\infty^+(M,g)$ of bounded sectional curvature, where $\dim M$ is odd, and let $(\widetilde{W},\tilde{g})$ be the elongation of $(W,g_W)$. Let $\slaD_{\slaS_0},\slaD_{\slaS_1}$ and $\slaD_{\slaS_{\widetilde{W}}}$ be the corresponding spin Dirac operators on $(M,g_0),(M,g_1)$ and $(\widetilde{W},\tilde{g})$, respectively. Then
\begin{equation}\label{E:cobor rel-eta}
\ind\slaD_{\slaS_{\widetilde{W}}}^+=\int_W\hatA(W,g_W)+\frac{1}{2}\eta(\slaD_{\slaS_1},\slaD_{\slaS_0}),
\end{equation}
where $\hatA(W,g_W)$ is the $\hatA$-genus form defined in \eqref{E:A-hat form}. In particular, if $(W,g_W)$ is a PSC-cobordism between $g_0$ and $g_1$, then
\[
\int_W\hatA(W,g_W)+\frac{1}{2}\eta(\slaD_{\slaS_1},\slaD_{\slaS_0})=0.
\]
\end{theorem}

\begin{remark}\label{R:cobor rel-eta}
Since $(W,g_W)$ has product structure outside a compact subset, the top-degree part of $\hatA(W,g_W)$ is compactly supported. Thus the integral in \eqref{E:cobor rel-eta} is well-defined.
\end{remark}

The proof of Theorem~\ref{T:cobor rel-eta} occupies the remaining of this section. The first ingredient is Lemma~\ref{L:cobor rel-eta}, which is a direct consequence of Gromov--Lawson's relative index theorem \cite[Theorem~4.18]{GromovLawson83} and B\"ar--Ballmann \cite[Theorem~1.21]{BaerBallmann12}.

\begin{lemma}\label{L:cobor rel-eta}
Under the hypothesis of Theorem~\ref{T:cobor rel-eta}, the quantity
\[
\ind\slaD_{\slaS_{\widetilde{W}}}^+-\int_W\hatA(W,g_W)
\]
depends only on $g_0$ and $g_1$.
\end{lemma}

\subsection{\texorpdfstring{$L^2$-}{L2-}index and the spectral flow on a cylinder}\label{SS:L2 ind spf}

For $g_0,g_1\in\frakR_\infty^+(M,g)$, let $g_r$, $r\in[0,1]$ be a smooth path of Riemannian metrics (not necessarily PSC metrics) connecting $g_0$ and $g_1$ such that each $g_r$ agrees with $g$ at infinity. Fix a smooth non-decreasing function $\kappa:[0,1]\to[0,1]$ such that $\kappa(r)=0$ for $r\le1/3$ and $\kappa(r)=1$ for $r\ge2/3$. We form a family of spin Dirac operators
\[
\slaD_{\slaS_{\kappa(r)}}:\dom\slaD_{\slaS_{\kappa(r)}}\to L^2(\slaS_{\kappa(r)}),\quad 0\le r\le1
\]
on $M$, where $\dom\slaD_{\slaS_{\kappa(r)}}$ is endowed with the graph norm. These operators have varying domains, as the structures of spinor bundles depend on the metric. However, as discussed in Bourguignon--Gauduchon \cite[Section~2]{BG92spinor} and Bandara--McIntosh--Ros\'en \cite[Section~3]{BMR18}, there is a fiberwise unitary map $\slaU_{\kappa(r)}:L^2(\slaS_0)\to L^2(\slaS_{\kappa(r)})$ for each $r\in[0,1]$, which is identity outside the compact subset where the metrics differ. Using conjugation by $\slaU_{\kappa(r)}$, we can view $\slaD_{\slaS_{\kappa(r)}}$ as a family of general elliptic first-order differential operators on a fixed domain $\dom\slaD_{\slaS_0}$. (In fact, the graph norms of $\slaD_{\slaS_{\kappa(r)}}$ for different $r$ are equivalent, see \cite[Section~3]{BG92spinor}.) They are Riesz continuous by \cite{BMR18}, thus the spectral flow $\spf(\slaD_{\slaS_{\kappa(r)}})_{[0,1]}$ can be defined.

From above, we have a concordance between $g_0$ and $g_1$. Precisely, equip $W:=M\times[0,1]$ with the metric $g_{\kappa(r)}+dr^2$, where $r$ denotes the coordinate along $[0,1]$. As in last subsection, we can consider the $\ZZ_2$-graded spin Dirac operators $\slaD_{\slaS_W}$ on $W$ and $\slaD_{\slaS_{\widetilde{W}}}$ on the elongation $\widetilde{W}=M\times\RR$. In this case, we have
\begin{equation}\label{E:Dirac elonga}
\slaD_{\slaS_{\widetilde{W}}}^+=\rmc(\nu)(\p_r+\slaD_{\slaS_{\tilde{\kappa}(r)}}),\qquad
\slaD_{\slaS_{\widetilde{W}}}^-=-(-\p_r+\slaD_{\slaS_{\tilde{\kappa}(r)}})\rmc(\nu),
\end{equation}
where $\nu$ is the unit normal vector field in the direction of $r$ and $\tilde{\kappa}:\RR\to[0,1]$ is the smooth extension of $\kappa$ such that $\tilde{\kappa}(r)=0$ for $r<0$ and $\tilde{\kappa}(r)=1$ for $r>1$. In this case, the $L^2$-index of $\slaD_{\slaS_{\widetilde{W}}}^+$ can be computed by a spectral flow.

\begin{theorem}\label{T:L2 ind spf}
Under the above setting, it holds that
\[
\ind\slaD_{\slaS_{\widetilde{W}}}^+=\spf(\slaD_{\slaS_{\kappa(r)}})_{[0,1]}.
\]
\end{theorem}

\begin{remark}\label{R:L2 ind spf}
As before, $\dom\slaD_{\slaS_{\widetilde{W}}}\subset L^2(\slaS_{\widetilde{W}})$ is endowed with the graph norm of $\slaD_{\slaS_{\widetilde{W}}}$. In view of \eqref{E:Dirac elonga}, it is equivalent to the graph norm of $\pm\p_r+\slaD_{\slaS_{\tilde{\kappa}(r)}}$. By the variation of spin Dirac operators associated to different metrics (cf. \cite[Theorem~3.3]{BG92spinor}), it is again equivalent to the graph norm of $\pm\p_r+\slaD_{\slaS_0}$. From this perspective, we can view $\slaD_{\slaS_{\widetilde{W}}}$ as an elliptic first-order differential operator mapping from $H^1(\RR,L^2(\slaS_0))\cap L^2(\RR,\dom\slaD_{\slaS_0})$ to $L^2(\RR,L^2(\slaS_0))$. The norm on the domain is given by
\[\label{E:cylinder norm}
\|s\|^2_{\slaD_{\slaS_{\widetilde{W}}}}=\int_{-\infty}^\infty\big(\|s\|^2_{\slaD_{\slaS_0}}+\|\p_rs\|^2_{L^2(\slaS_0)}\big)dr,
\]
for $s\in H^1(\RR,L^2(\slaS_0))\cap L^2(\RR,\dom\slaD_{\slaS_0})$. This is what we will need later in this subsection.
\end{remark}

Results in the form of Theorem~\ref{T:L2 ind spf} appear often in the literature on index theory, and have been studied extensively under different conditions, especially in compact situations. Here we follow the strategy of Robbin--Salamon \cite[Section~4]{RobbinSalamon95} to prove this theorem in steps.

\begin{lemma}\label{L:homotop inv}
Let $g'_r$, $r\in[0,1]$ be another smooth path of metrics connecting $g_0$ and $g_1$ such that each $g'_r$ agrees with $g$ at infinity. Denote the corresponding spin Dirac operators by $\slaD_{\slaS_{\widetilde{W}}}'$ and $\slaD_{\slaS_{\kappa(r)}}'$. If $g'_r$ is homotopic to $g_r$, then
\[
\ind\slaD_{\slaS_{\widetilde{W}}}'^+=\ind\slaD_{\slaS_{\widetilde{W}}}^+,\qquad
\spf(\slaD_{\slaS_{\kappa(r)}}')_{[0,1]}=\spf(\slaD_{\slaS_{\kappa(r)}})_{[0,1]}.
\]
\end{lemma}

\begin{proof}
The equation about the spectral flow follows from the homotopy invariance of the spectral flow (cf. Booss-Bavnbek--Lesch--Phillips \cite[Proposition~2.3]{BoossLeschPhillips05}). For the index, notice that the metric defining $\slaD_{\slaS_{\widetilde{W}}}'$ is a compact perturbation of that defining $\slaD_{\slaS_{\widetilde{W}}}$. Again viewed as operators acting on a fixed domain, there exists a Riesz continuous family of operators connecting $\slaD_{\slaS_{\widetilde{W}}}$ and $\slaD_{\slaS_{\widetilde{W}}}'$. Hence there exists a norm-continuous family of Fredholm operators connecting $\frac{\slaD_{\slaS_{\widetilde{W}}}}{\sqrt{1+(\slaD_{\slaS_{\widetilde{W}}})^2}}$ and $\frac{\slaD_{\slaS_{\widetilde{W}}}'}{\sqrt{1+(\slaD_{\slaS_{\widetilde{W}}}')^2}}$, and the equation about the index follows from the homotopy invariance of the Fredholm index. (Alternatively, one can also use the relative index theory to get the proof.)
\end{proof}

From this lemma, by perturbing the path (and applying the Sard--Smale theorem) if necessary, we can assume henceforth that the set of points $r\in[0,1]$ where $\slaD_{\slaS_{\kappa(r)}}$ has a non-trivial kernel is discrete (thus finite), and the kernel is one-dimensional at each such point. Notice that the spectrum of each $\slaD_{\slaS_{\kappa(r)}}$ near zero consists of only eigenvalues with finite multiplicities.

\begin{lemma}\label{L:direct sum homotop}
Suppose the eigenvalues of $\slaD_{\slaS_{\kappa(r)}}$ cross zero $m$ times as $r$ varies from 0 to 1. Then there exists a path of Hermitian matrices $B(r)\in\End(\CC^m)$ such that the path
\[
\slaD_{\slaS_{\kappa(r)}}\oplus B(r):\dom\slaD_{\slaS_0}\oplus\CC^m\to L^2(\slaS_0)\oplus\CC^m,\quad 0\le r\le1
\]
is homotopic to a path of self-adjoint invertible operators.
\end{lemma}

\begin{proof}
Suppose the eigenvalues of $\slaD_{\slaS_{\kappa(r)}}$ cross zero at $r_0\in[1/3,2/3]$. By Bourguignon--Gauduchon \cite[Section~3]{BG92spinor}, the eigenvalues and eigenvectors of $\slaD_{\slaS_{\kappa(r)}}$ vary differentiably. In a small interval $(r_0-\varepsilon,r_0+\varepsilon)$ around $r_0$, let $\lambda(r)$ be the crossing eigenvalue which is non-zero at points other than $r_0$ and let $\psi(r)$ be the corresponding unit eigenvector. Define a bounded operator $\pi(r):L^2(\slaS_0)\to\CC$ by
\[
\pi(r)s:=\langle\psi(r),s\rangle_{L^2(\slaS_0)}.
\]

Choose a smooth cut-off function $\beta:[0,1]\to[0,1]$ such that $\beta(r)=1$ for $|r-r_0|\le\varepsilon/2$ and $\beta(r)=0$ for $|r-r_0|\ge\varepsilon$. Choose a $C^1$-function $b:[0,1]\to\RR$ such that $b(r)=-\lambda(r)$ for $|r-r_0|\le\varepsilon/2$, $b(r)\ne0$ for $r\ne r_0$, and $b(r)$ is constant for $|r-r_0|\ge\varepsilon$. Moreover, on $(r_0-\varepsilon,r_0+\varepsilon)$, we can require that $|b(r)+\lambda(r)|\le\delta/2$, where $\delta$ is the positive infimum of $|\lambda(r)|$ for $\varepsilon/2<|r-r_0|<\varepsilon$. Then one can construct a homotopy
\begin{equation}\label{E:direct sum homotop}
A_\theta(r):=\left(
\begin{matrix}
\slaD_{\slaS_{\kappa(r)}} & \theta\beta(r)\pi(r)^* \\
\theta\beta(r)\pi(r) & b(r)
\end{matrix}
\right),\quad 0\le\theta\le1
\end{equation}
between the paths $\slaD_{\slaS_{\kappa(r)}}\oplus b(r)=A_0(r)$ and $A_1(r)$. As long as we can show that $A_1(r)$ is invertible near $r_0$ and the eigenvalues of $A_1(r)$ cross zero $m-1$ times, the required path $B(r)$ can be obtained by doing this inductively.

In fact, outside $(r_0-\varepsilon,r_0+\varepsilon)$, for each $\theta$, the operator $A_\theta(r)$ is diagonal with $b(r)$ a non-zero constant. So the eigenvalues of $A_\theta(r)$ cross zero if and only if the eigenvalues of $\slaD_{\slaS_{\kappa(r)}}$ cross zero.

We make a detailed discussion about the invertibility of $A_1(r)$ on $(r_0-\varepsilon,r_0+\varepsilon)$. For any $s\oplus x\in\dom\slaD_{\slaS_0}\oplus\CC$, we write $s=s'\oplus s''$, where $s'\in\span_\CC\{\psi(r)\}$ and $s''$ is $L^2$-orthogonal to $\psi(r)$. Noticing that $\pi(r)^*x=x\psi(r)$, we compute
\[
\begin{aligned}
& \|A_1(r)(s\oplus x)\|^2_{L^2(\slaS_0)\oplus\CC} \\
& \qquad =\|\slaD_{\slaS_{\kappa(r)}}s+\beta(r)\pi(r)^*x\|^2_{L^2(\slaS_0)}+\big|\beta(r)\pi(r)s+b(r)x\big|^2 \\
& \qquad =\|\slaD_{\slaS_{\kappa(r)}}s\|^2_{L^2(\slaS_0)}+\beta^2(r)|x|^2+2\beta(r)\Re\big[\bar{x}\pi(r)(\slaD_{\slaS_{\kappa(r)}}s)\big] \\
& \qquad \quad+\beta^2(r)|\pi(r)s|^2+b^2(r)|x|^2+2\beta(r)b(r)\Re[\bar{x}\pi(r)s] \\
& \qquad =\lambda^2(r)\|s'\|^2_{L^2(\slaS_0)}+\|\slaD_{\slaS_{\kappa(r)}}s''\|^2_{L^2(\slaS_0)}+\beta^2(r)|x|^2 \\
& \qquad \quad+\beta^2(r)|\pi(r)s|^2+b^2(r)|x|^2+2\beta(r)\big(\lambda(r)+b(r)\big)\Re[\bar{x}\pi(r)s].
\end{aligned}
\]  
If $|r-r_0|\le\varepsilon/2$, then
\[
\|A_1(r)(s\oplus x)\|^2_{L^2(\slaS_0)\oplus\CC}\ge c_1\|s''\|^2_{L^2(\slaS_0)}+|x|^2+\|s'\|^2_{L^2(\slaS_0)}\ge c_2\|s\oplus x\|^2_{L^2(\slaS_0)\oplus\CC}
\]
for some constants $c_1>0$ (determined by the positive lower bound of the absolute values of the non-zero spectra of $\slaD_{\slaS_{\kappa(r)}}$) and $c_2=\min\{c_1,1\}>0$. If $\varepsilon/2<|r-r_0|<\varepsilon$, then
\[
\begin{aligned}
\|A_1(r)(s\oplus x)\|^2_{L^2(\slaS_0)\oplus\CC}& \ge\delta^2\|s'\|^2_{L^2(\slaS_0)}+c_1\|s''\|^2_{L^2(\slaS_0)}+\beta^2(r)|x|^2 \\
& \quad+b^2(r)|x|^2-\delta\beta(r)\Re[\bar{x}\pi(r)s] \\
& \ge\frac{3\delta^2}{4}\|s'\|^2_{L^2(\slaS_0)}+c_1\|s''\|^2_{L^2(\slaS_0)}+\delta'^2|x|^2 \\
& \quad+\frac{\delta^2}{4}|\pi(r)s|^2+\beta^2(r)|x|^2-\delta\beta(r)\Re[\bar{x}\pi(r)s] \\
& \ge c_3\|s\oplus x\|^2_{L^2(\slaS_0)\oplus\CC},
\end{aligned}
\]
where $\delta>0$ is given earlier, $\delta':=\inf_{\varepsilon/2<|r-r_0|<\varepsilon}b(r)>0$, and $c_3=\min\{c_1,3\delta^2/4,\delta'^2\}\\ >0$. In summary, it is proved that when $r\in(r_0-\varepsilon,r_0+\varepsilon)$, one always has
\[
\|A_1(r)(s\oplus x)\|^2_{L^2(\slaS_0)\oplus\CC}\ge c\|s\oplus x\|^2_{L^2(\slaS_0)\oplus\CC}
\]
for some constant $c>0$ that is independent of $r$, so that $A_1(r)$ is invertible.
\end{proof}

In view of Lemma~\ref{L:direct sum homotop}, we obtain
\[
\spf(\slaD_{\slaS_{\kappa(r)}}\oplus B(r))_{[0,1]}=0
\]
by the homotopy invariance of the spectral flow. We also want to consider the corresponding index. Note that now $\slaD_{\slaS_{\kappa(r)}}\oplus B(r)$ is no longer a family of operators on $M$. Nevertheless, by \eqref{E:Dirac elonga}, it is equivalent to consider
\[
\begin{aligned}
\ind\big(\p_r+\slaD_{\slaS_{\tilde{\kappa}(r)}}\oplus\tilde{B}(r)\big):=& \dim\ker\big(\p_r+\slaD_{\slaS_{\tilde{\kappa}(r)}}\oplus\tilde{B}(r)\big) \\
& -\dim\ker\big(-\p_r+\slaD_{\slaS_{\tilde{\kappa}(r)}}\oplus\tilde{B}(r)\big),
\end{aligned}
\]
where $\tilde{B}(r)$ is the natural extension of $B(r)$ to $r\in\RR$ such that it is constant for $r<0$ and $r>1$. Here as in Remark~\ref{R:L2 ind spf}, we view $\pm\p_r+\slaD_{\slaS_{\tilde{\kappa}(r)}}\oplus\tilde{B}(r)$ as operators mapping from $(H^1(\RR,L^2(\slaS_0))\cap L^2(\RR,\dom\slaD_{\slaS_0}))\oplus H^1(\RR,\CC^m)$ to $L^2(\RR,L^2(\slaS_0))\oplus L^2(\RR,\CC^m)$. 

\begin{lemma}\label{L:direct sum ind=0}
We have
\[
\ind\big(\p_r+\slaD_{\slaS_{\tilde{\kappa}(r)}}\oplus\tilde{B}(r)\big)=0.
\]
\end{lemma}

\begin{proof}
Without loss of generality, we can assume $m=1$ in Lemma~\ref{L:direct sum homotop}. For $\theta\in[0,1]$, let $\tilde{A}_\theta(r)$ be the natural extension of $A_\theta(r)$ as in \eqref{E:direct sum homotop} to $r\in\RR$. Then there is a norm-continuous family of operators $\p_r+\tilde{A}_\theta(r)$, $\theta\in[0,1]$ connecting $\p_r+\slaD_{\slaS_{\tilde{\kappa}(r)}}\oplus\tilde{B}(r)=\p_r+\tilde{A}_0(r)$ and $\p_r+\tilde{A}_1(r)$. Note that the bounded operator
\[
\beta(r)\pi(r):H^1(\RR,L^2(\slaS_0))\cap L^2(\RR,\dom\slaD_{\slaS_0})\to L^2(\RR,\CC)
\]
is actually the composition of the following maps
\begin{multline*}
H^1(\RR,L^2(\slaS_0))\cap L^2(\RR,\dom\slaD_{\slaS_0})\xrightarrow{\beta(r)\pi(r)}H^1([r_0-\varepsilon,r_0+\varepsilon],\CC) \\\hookrightarrow L^2([r_0-\varepsilon,r_0+\varepsilon],\CC)\hookrightarrow L^2(\RR,\CC).
\end{multline*}
The second map is a compact embedding by the Rellich embedding theorem. Hence $\beta(r)\pi(r)$, as well as $\beta(r)\pi(r)^*$, is a compact operator. This shows that $\p_r+\tilde{A}_\theta(r)$ is a compact perturbation of $\p_r+\slaD_{\slaS_{\tilde{\kappa}(r)}}\oplus\tilde{B}(r)$ (which is Fredholm). One then concludes that $\p_r+\tilde{A}_\theta(r)$ is a norm-continuous family of Fredholm operators.

To prove the lemma, we consider the new family of operators $\p_r+\tilde{A}_\theta(\frac{r}{a})$, $\theta\in[0,1]$, with $a>0$ a constant to be determined. This corresponds to rescaling the width of the concordance to $a$, with the metric $g_{\kappa(\frac{r}{a})}+dr^2$ on $M\times[0,a]$. By doing this, the end-point values are preserved. By the same reason as the above discussion, $\p_r+\tilde{A}_\theta(\frac{r}{a})$, $\theta\in[0,1]$ is a norm-continuous family of Fredholm operators, thus having the same index.
We examine the operator
\[
\p_rA_1\Big(\frac{r}{a}\Big)=\frac{1}{a}\left(
\begin{matrix}
\dot{\slaD}_{\slaS_{\kappa(\frac{r}{a})}} & \dot{\beta}(\frac{r}{a})\pi(\frac{r}{a})^*+\beta(\frac{r}{a})\dot{\pi}(\frac{r}{a})^* \\
\dot{\beta}(\frac{r}{a})\pi(\frac{r}{a})+\beta(\frac{r}{a})\dot{\pi}(\frac{r}{a}) & \dot{b}(\frac{r}{a})
\end{matrix}
\right),\quad 0\le r\le a,
\]
where $\dot{\alpha}$ denotes the differentiation of $\alpha$. By Bourguignon--Gauduchon \cite[Theorem~3.3]{BG92spinor}, $\dot{\slaD}_{\slaS_{\kappa(\frac{r}{a})}}$ is a first-order differential operator supported on a compact subset of $M$. The other entries are $L^2$-bounded operators. It can be seen that for any small $\epsilon>0$, one can choose $a$ large enough such that the graph norm of $\p_r\tilde{A}_1(\frac{r}{a})$ is dominated by $\epsilon$ times the graph norm of $\tilde{A}_1(\frac{r}{a})$ for all $r\in\RR$.

\medskip
\noindent\emph{Claim.} For $a$ large enough, $\pm\p_r+\tilde{A}_1(\frac{r}{a})$ are invertible operators.

For simplicity, we denote $\calW:=(H^1(\RR,L^2(\slaS_0))\cap L^2(\RR,\dom\slaD_{\slaS_0}))\oplus H^1(\RR,\CC)$ and $\calH:=L^2(\RR,L^2(\slaS_0))\oplus L^2(\RR,\CC)$. Then for any $\tilde{s}(r)=s(r)\oplus x(r)\in\calW$,
\[
\begin{aligned}
& \left\|\left(\p_r+\tilde{A}_1\Big(\frac{r}{a}\Big)\right)\tilde{s}\right\|^2_\calH \\
& \qquad =\|\p_r\tilde{s}\|^2_\calH+\left\|\tilde{A}_1\Big(\frac{r}{a}\Big)\tilde{s}\right\|^2_\calH-\int_{-\infty}^\infty\left\langle\left(\p_r\tilde{A}_1\Big(\frac{r}{a}\Big)\right)\tilde{s},\tilde{s}\right\rangle_{L^2(\slaS_0)\oplus\CC}dr \\
& \qquad \ge\left\|\tilde{A}_1\Big(\frac{r}{a}\Big)\tilde{s}\right\|^2_\calH-\|\tilde{s}\|_\calH\left\|\left(\p_r\tilde{A}_1\Big(\frac{r}{a}\Big)\right)\tilde{s}\right\|_\calH \\
& \qquad \ge\left\|\tilde{A}_1\Big(\frac{r}{a}\Big)\tilde{s}\right\|^2_\calH-\epsilon\|\tilde{s}\|_\calH\left(\|\tilde{s}\|^2_\calH+\left\|\tilde{A}_1\Big(\frac{r}{a}\Big)\tilde{s}\right\|^2_\calH\right)^{1/2} \\
& \qquad \ge\left\|\tilde{A}_1\Big(\frac{r}{a}\Big)\tilde{s}\right\|^2_\calH-\epsilon\|\tilde{s}\|_\calH\left\|\tilde{A}_1\Big(\frac{r}{a}\Big)\tilde{s}\right\|_\calH-\epsilon\|\tilde{s}\|^2_\calH.
\end{aligned}
\]
Recall in the proof of Lemma~\ref{L:direct sum homotop}, we have shown that $\|\tilde{A}_1(\frac{r}{a})\tilde{s}\|^2_\calH\ge c\|\tilde{s}\|^2_\calH$. Hence, for $a$ large enough, one can make $\epsilon$ small enough so that $\|(\p_r+\tilde{A}_1(\frac{r}{a}))\tilde{s}\|^2_\calH\ge C\|\tilde{s}\|^2_\calH$ for some constant $C>0$. Therefore, $\p_r+\tilde{A}_1(\frac{r}{a})$ is invertible. The other one is exactly the same, and the claim is proved.

From this claim, we can fix a large enough $a$ such that
\[
\ind\left(\p_r+\slaD_{\slaS_{\tilde{\kappa}(\frac{r}{a})}}\oplus\tilde{B}\Big(\frac{r}{a}\Big)\right)=\ind\left(\p_r+\tilde{A}_1\Big(\frac{r}{a}\Big)\right)=0.
\]
The term on the left-hand side is the sum of $\ind(\p_r+\slaD_{\slaS_{\tilde{\kappa}(\frac{r}{a})}})$ and $\ind(\p_r+\tilde{B}(\frac{r}{a}))$. It is clear that $\ind(\p_r+\tilde{B}(\frac{r}{a}))=\ind(\p_r+\tilde{B}(r))$. For the first one, note that by \eqref{E:Dirac elonga}, $\ind(\p_r+\slaD_{\slaS_{\tilde{\kappa}(\frac{r}{a})}})$ and $\ind(\p_r+\slaD_{\slaS_{\tilde{\kappa}(r)}})$ represent respectively indices of two spin Dirac operators on the elongations of two concordances between $g_0$ and $g_1$. Note that the metrics on these two elongations only differ on a compact subset. Arguing as in the proof of Lemma~\ref{L:homotop inv}, the two indices are equal. This implies that $\ind(\p_r+\slaD_{\slaS_{\tilde{\kappa}(\frac{r}{a})}})=\ind(\p_r+\slaD_{\slaS_{\tilde{\kappa}(r)}})$. With these combined, the lemma follows.
\end{proof}

\begin{remark}\label{R:direct sum ind=0}
In fact, $\ind(\p_r+\tilde{B}(r))$ can be computed directly. We know from the proof of Lemma~\ref{L:direct sum homotop} that if the eigenvalues of $\slaD_{\slaS_{\kappa(r)}}$ cross zero $m$ times as $r$ varies from 0 to 1, then $\tilde{B}(r)$ is a direct sum of $m$ real-valued functions $\tilde{b}(r)$, each of which is constant outside $[0,1]$. In this case, $\ind(\p_r+\tilde{B}(r))$ is equal to the difference of the number of those functions which are negative for $r<0$ and positive for $r>1$ and the number of those functions which are positive for $r<0$ and negative for $r>1$. By the way of constructing $b(r)$ in Lemma~\ref{L:direct sum homotop}, this is exactly opposite of the spectral flow of $\slaD_{\slaS_{\kappa(r)}}$, $r\in[0,1]$. So
\[
\ind\big(\p_r+\tilde{B}(r)\big)=-\spf(\slaD_{\slaS_{\kappa(r)}})_{[0,1]}.
\]
\end{remark}

Now for a path of metrics $g_r$, $r\in[0,1]$ satisfying the assumption above Lemma \ref{L:direct sum homotop}, Theorem~\ref{T:L2 ind spf} is a consequence of \eqref{E:Dirac elonga}, Lemma~\ref{L:direct sum ind=0} and Remark~\ref{R:direct sum ind=0}. For a general path, Theorem~\ref{T:L2 ind spf} follows from Lemma~\ref{L:homotop inv}.

\subsection{Proof of Theorem~\ref{T:cobor rel-eta}}\label{SS:pf APS ind thm}

With the help of Theorem~\ref{T:L2 ind spf}, the index on the elongation is transferred to a spectral flow. We can then use the results in Section~\ref{S:sp flow} to complete the proof of Theorem~\ref{T:cobor rel-eta}.

\begin{proof}[Proof of Theorem~\ref{T:cobor rel-eta}]
Let $g_0$ and $g_1$ be as in Theorem~\ref{T:cobor rel-eta}. By Lemma~\ref{L:cobor rel-eta}, it suffices to prove that formula \eqref{E:cobor rel-eta} holds for the $\ZZ_2$-graded spin Dirac operator $\slaD_{\slaS_{\widetilde{W}}}$ on the elongation $\widetilde{W}$ of $W:=M\times[0,1]$ constructed in last subsection. By Theorem~\ref{T:L2 ind spf}, it reduces to proving that
\begin{equation}\label{E:spf=rel eta}
\spf(\slaD_{\slaS_{\kappa(r)}})_{[0,1]}=\int_W\hatA(W,g_W)+\frac{1}{2}\eta(\slaD_{\slaS_1},\slaD_{\slaS_0}).
\end{equation}

As mentioned in the beginning of Subsection \ref{SS:L2 ind spf}, here $\slaD_{\slaS_{\kappa(r)}}$ is actually identified with the conjugation by $\slaU_{\kappa(r)}$. Strictly speaking, it is no longer a Dirac-type operator on $\dom\slaD_{\slaS_0}$. But it is not hard to check that the asymptotic expansion as in Lemma~\ref{L:var trun rel-eta} still exists, and the proof of Proposition~\ref{P:sp-flow trun rel-eta} still holds (indeed, the proofs of \cite[Theorem~5.8, Proposition~5.10]{Shi22} still hold) in this case. Note that both $\slaD_{\slaS_0}$ and $\slaD_{\slaS_1}$ are invertible, and the relative eta invariant does not change under conjugation by $\slaU_{\kappa(r)}$ (cf. Remark~\ref{R:rel-eta}). By Proposition~\ref{P:sp-flow trun rel-eta},
\[
\spf(\slaD_{\slaS_{\kappa(r)}})_{[0,1]}=\Big(\frac{\varepsilon}{\pi}\Big)^{1/2}\int_0^1\Tr\big(\dot{\slaD}_{\slaS_{\kappa(r)}}e^{-\varepsilon\slaD_{\slaS_{\kappa(r)}}^2}\big)dr+\frac{1}{2}\eta_\varepsilon(\slaD_{\slaS_1},\slaD_{\slaS_0}).
\]
So \eqref{E:spf=rel eta} is finally reduced to
\begin{equation}\label{E:lim spf=rel eta}
\lim_{\varepsilon\to0}\Big(\frac{\varepsilon}{\pi}\Big)^{1/2}\int_0^1\Tr\big(\dot{\slaD}_{\slaS_{\kappa(r)}}e^{-\varepsilon\slaD_{\slaS_{\kappa(r)}}^2}\big)dr=\int_W\hatA(W,g_W).
\end{equation}
A crucial point of this equation is that both sides are \emph{local}. So again as we did in the proof of Theorem~\ref{T:sp flow-trans}, one can use the locality to replace $\dot{\slaD}_{\slaS_{\kappa(r)}}e^{-\varepsilon\slaD_{\slaS_{\kappa(r)}}^2}$ by the corresponding one on the closed double $\widehat{M}$ of the compact subset of $M$ where the metrics differ. In this way the question is equivalent to the classical APS index theorem on the compact manifold $\widehat{M}\times[0,1]$, which certainly holds. Hence \eqref{E:lim spf=rel eta} is true and Theorem~\ref{T:cobor rel-eta} is proved.
\end{proof}

The following is a twisted version of Theorem~\ref{T:cobor rel-eta}.

\begin{corollary}\label{C:cobor twisted}
Let the hypothesis be as in Theorem~\ref{T:cobor rel-eta}. Suppose $F\to M$ is a unitary flat bundle that extends to a unitary flat bundle $F_W$ over $W$. Let $F_{\widetilde{W}}$ be the obvious extension of $F_W$ to $\widetilde{W}$. Then
\[
\ind\slaD_{\slaS_{\widetilde{W}}\otimes F_{\widetilde{W}}}^+=\int_W\hatA(W,g_W)\cdot\rank(F_{\widetilde{W}})+\frac{1}{2}\eta(\slaD_{\slaS_1\otimes F},\slaD_{\slaS_0\otimes F}).
\]
In particular, if $(W,g_W)$ is a PSC-cobordism between $g_0$ and $g_1$, then
\begin{equation}\label{E:cobor twisted}
\int_W\hatA(W,g_W)\cdot\rank(F_{\widetilde{W}})+\frac{1}{2}\eta(\slaD_{\slaS_1\otimes F},\slaD_{\slaS_0\otimes F})=0.
\end{equation}
\end{corollary}

\section{Space of uniformly PSC metrics on connected sums}\label{S:psc conn-sum}

It is known that eta invariants have important applications in studying positive scalar curvature problems. More precisely, they can be used to investigate the topology of the space of PSC metrics on compact manifolds. In this section, using the index formula derived in Section \ref{S:ind psc}, we shall prove some disconnectivity results about the (moduli) spaces of uniformly PSC metrics on non-compact connected sums.

We first recall some general results along this line for closed manifolds. Roughly speaking, in low dimensions (dimension 2 or 3), the (moduli) space of PSC metrics (when non-empty) is path-connected (even contractible) (cf. Rosenberg--Stolz \cite{RosenbergStolz01}, Marques \cite{Marques12},  Bamler--Kleiner \cite{BamlerKleiner19}); while in high dimensions (dimension $\ge4$), the space is disconnected in many cases. Such results include Hitchin \cite{Hitchin74}, Carr \cite{Carr88}, Botvinnik--Gilkey \cite{BotviGilkey95eta-psc,BotviGilkey96psc-spaceform}, Ruberman \cite{Ruberman98,Ruberman01}, Piazza--Schick \cite{PiazzaSchick07torsion-rho}, Mrowka--Ruberman--Saveliev \cite{MRS16}, etc. For more details, see Tuschmann--Wraith \cite{TuschmannWraith15book} and Carlotto \cite{Carlotto21}.

When the manifold is non-compact, there are some results in low dimensions recently, cf. Belegradek--Hu \cite{BeleHu15}, Bessi\`eres--Besson--Maillot--Marques \cite{BBMM21}. But in high dimensions, little is known. In this paper we restrict to the space of uniformly PSC metrics which coincide at infinity discussed in Subsection~\ref{SS:space psc}. In this case, some results mentioned above can be extended.

\subsection{Uniformly PSC metrics on connected sums}\label{SS:psc conn-sum}

Recall that a connected sum $M=N\#N'$ of two manifolds $N$ and $N'$ of the same dimension (without boundary) can be written as
\begin{equation}\label{E:conn-sum}
M=\mathring{N}\cup_{S\times I}\mathring{N}',
\end{equation}
where $\mathring{N}$ (resp. $\mathring{N}'$) is $N$ (resp. $N'$) with an open ball removed, $S$ is the common boundary of the balls (under certain identification), and $S\times I$ is the ``neck" part. As a special case, the connected sum of any manifold $M$ with a sphere of the same dimension is diffeomorphic to $M$.

Suppose $h$ (resp. $h'$) is a complete uniformly PSC metric on $N$ (resp. $N'$). Then by Gromov--Lawson \cite{GromovLawson80classification}, one can form a complete uniformly PSC metric $h\#h'$ on $M=N\#N'$. To be precise, one can deform the metric in a small ball around a given point preserving positive scalar curvature such that the metric near that point becomes the Riemannian product $\RR\times S(\varepsilon)$, where $S(\varepsilon)$ is the standard sphere in Euclidean space of radius $\varepsilon$ (which is small enough). After doing this deformation to both $(N,h)$ and $(N',h')$, one can paste them together along the cylindrical end. It should be pointed out that although the way to construct $h\#h'$ is not unique, they  will all coincide at infinity and lie in the same path component of $\frakR_\infty^+(M,h\#h')$. 

\begin{remark}\label{R:conn-sum}
For later use, $N$ will be a closed manifold. From the above discussion, it is convenient to use the notation $\frakR_\infty^+(M,h')$ to denote the space of complete uniformly PSC metrics on $M$ which coincide with $h\#h'$ at infinity, where $h$ can be any PSC metric on $N$. More generally, we can also consider the case that $N'$ has compact boundary, as long as the connected sum is taken away from the boundary. In this situation, we use $\frakR_{\infty,\pM}^+(M,h')$ (or just $\frakR_{\pM}^+(M,h')$ if $N'$ is compact) to denote the space of complete uniformly PSC metrics on $M$ which coincide with $h\#h'$ at infinity and near the boundary, with $h$ being any PSC metric on $N$.
\end{remark}

\subsection{Non-isotopic PSC metrics in dimensions \texorpdfstring{$4m-1$}{4m-1} with \texorpdfstring{$m\geq 2$}{m≥2}}\label{SS:4m-1}

In \cite{Carr88}, Carr shows that the space of PSC metrics on the $(4m-1)$-sphere $S^{4m-1}$ has infinitely many path components for $m\ge2$. As pointed out in Lawson--Michelsohn \cite[\S IV.7]{LawMic89}, Carr's argument works for any closed spin $(4m-1)$-manifolds which admit a PSC metric.

When considering non-compact manifolds, the argument in the closed case can be straightforwardly repeated to show the following.

\begin{proposition}\label{P:non-iso 4m-1}
Let $M$ be a non-compact spin $(4m-1)$-manifold (without boundary) which admits a uniformly PSC metric $g$. Then 
$\pi_0(\frakR_\infty^+(M,g))$ is infinite.
\end{proposition}

For the convenience of the reader, we give a sketch of the above-mentioned argument here.

By a plumbing technique, Carr was able to construct a compact $4m$-manifold $Y_k$ for each $k\in\NN$ with $\p Y_k=S^{4m-1}$, such that $Y_k$ admits a PSC metric which is a product near the boundary. Moreover, let $X_{k,k'}=Y_k\cup_{S^{4m-1}}Y_{k'}$. Then the $\hat{A}$-genus $\hat{\mathrm{A}}(X_{k,k'})\ne0$ for $k\ne k'$. If $\gamma_k$ (resp. $\gamma_{k'}$) denotes the induced PSC metric on $S^{4m-1}=\p Y_k$ (resp. $\p Y_{k'}$), then one can show that $g\#\gamma_k$ and $g\#\gamma_{k'}$ belong to different path components of $\frakR_\infty^+(M,g)$ for $k\ne k'$.

Indeed, let $Z_k=(M\times[0,1])\natural Y_k$, where $\natural$ denotes boundary connected sum. Then the metric $g$ on $M$ and $g\#\gamma_k$ on $M\#S^{4m-1}\cong M$ extends to a PSC metric on $Z_k$ with product structure near the boundary. Let $Z_{k'}$ be built analogously. Clearly, $Z_k$ and $Z_{k'}$ can be glued along the ends $M\times\{0\}$. For $k\ne k'$, if $g\#\gamma_k$ and $g\#\gamma_{k'}$ are PSC-isotopic, then one can join the other two ends of $Z_k$ and $Z_{k'}$ by a PSC-concordance between $g\#\gamma_k$ and $g\#\gamma_{k'}$. In this way we get a uniformly PSC metric on $(M\times S^1)\#X_{k,k'}$. By Gromov--Lawson's relative index theorem,
\[
0=\ind\slaD^+_{\slaS_{(M\times S^1)\#X_{k,k'}}}=\hat{\mathrm{A}}(X_{k,k'})\ne0,
\]
which is a contradiction. Hence $\frakR_\infty^+(M,g)$ has infinitely many path components.

\begin{remark}\label{R:non-iso 4m-1}
Note that although these metrics are non-PSC-isotopic, they are actually PSC-cobordant.
\end{remark}

\subsection{The relative rho invariant}\label{SS:rel-rho}

There is another method from index theory in studying the topology of the space of PSC metrics. It uses eta invariants with coefficients in unitary flat bundles induced by representations of the fundamental group. They are sometimes called rho invariants.

Let $N$ be a closed spin manifold admitting a PSC metric. Suppose $N$ has a non-trivial fundamental group $\pi=\pi_1(N)$. Let $\lambda$ be a unitary representation of $\pi$. Then $\lambda$ defines a unitary flat bundle $F_\lambda^N:=(\widetilde{N}\times\CC^l)/\Gamma$ over $N$, where $\widetilde{N}$ is the universal cover of $N$, and $\Gamma$ is the action of $\pi$ given by
\[
\alpha\cdot(\tilde{x},v)=(\alpha\tilde{x},\lambda(\alpha)v),\qquad\forall\alpha\in\pi,\;\tilde{x}\in\widetilde{N},\;v\in\CC^l.
\]
We call $\lambda$ a virtual unitary representation of virtual dimension 0, if $\lambda$ is a formal difference of two finite dimensional unitary representations $\lambda^+$ and $\lambda^-$ of $\pi$ with $\dim\lambda^+=\dim\lambda^-$. Let $R_0(\pi)$ denote the set of virtual unitary representations of virtual dimension 0. For $\lambda=\lambda^+-\lambda^-\in R_0(\pi)$, the \emph{rho invariant} associated to $\lambda$ is
\[
\rho(\slaD_{\slaS_{N}})(\lambda):=\eta(\slaD_{\slaS_{N}\otimes F_{\lambda^+}^N})-\eta(\slaD_{\slaS_{N}\otimes F_{\lambda^-}^N}).
\]

Now consider the connected sum $M=N\#N'$, where $N'$ is a non-compact spin manifold admitting a uniformly PSC metric $h'$ of bounded geometry. Under a canonical projection $N\#N'\to N$, the pullback of $F_{\lambda^\pm}^N$, denoted by $F_{\lambda^\pm}$, are two unitary flat bundles over $M$. It is clear that $F_{\lambda^\pm}$ are trivial bundles over $\mathring{N}'$.
Let $g_0=h_0\#h'$ and $g_1=h_1\#h'$ be two uniformly PSC metrics of bounded geometry on $M$. Let $\slaD_{\slaS_0\otimes F_{\lambda^\pm}}$ and $\slaD_{\slaS_1\otimes F_{\lambda^\pm}}$ be the twisted spin Dirac operators. Then they are Dirac--Schr\"odinger operators which coincide at infinity, and with empty essential supports.

\begin{definition}\label{D:rel-rho}
Under the above setting, the \emph{relative rho invariant} associated to $(g_0,g_1)$ and $\lambda=\lambda^+-\lambda^-\in R_0(\pi)$ is defined to be
\[
\rho(\slaD_{\slaS_1},\slaD_{\slaS_0})(\lambda):=\eta(\slaD_{\slaS_1\otimes F_{\lambda^+}},\slaD_{\slaS_0\otimes F_{\lambda^+}})-\eta(\slaD_{\slaS_1\otimes F_{\lambda^-}},\slaD_{\slaS_0\otimes F_{\lambda^-}}).
\]
\end{definition}

Let $(W,g_W)$ be a PSC-cobordism between $g_0$ and $g_1$. If any $F_\lambda$ can be extended to a unitary flat bundle over $W$, then we call $g_0$ and $g_1$ \emph{$\pi_1$-PSC-cobordant}. This is satisfied for example if $W$ admits a $\pi_1(M)$-covering whose boundary is the union of the universal coverings of the two boundary components.
The following is an immediate consequence of Corollary~\ref{C:cobor twisted} by applying \eqref{E:cobor twisted} twice and taking difference.

\begin{proposition}\label{P:rel-rho cobor}
If $g_0$ and $g_1$ are $\pi_1$-PSC-cobordant, then $\rho(\slaD_{\slaS_1},\slaD_{\slaS_0})(\lambda)=0$ for any $\lambda\in R_0(\pi)$. In particular, if $g_0$ and $g_1$ are PSC-concordant, then $\rho(\slaD_{\slaS_1},\slaD_{\slaS_0})(\lambda)\\ =0$.
\end{proposition}

Let $N_0$ denote $(N,h_0)$ and $N_1$ denote $(N,h_1)$. On the closed manifold $N$, the relative rho invariant now corresponds to the difference of the two individual rho invariants
\[
\begin{aligned}
\rho(\slaD_{\slaS_{N_1}})(\lambda)-\rho(\slaD_{\slaS_{N_0}})(\lambda).
\end{aligned}
\]
It turns out that when taking connected sum with a fixed manifold (with a fixed uniformly PSC metric of bounded geometry), the relative rho invariant is unchanged. Namely, we have

\begin{proposition}\label{P:rel-rho conn-sum}
$\rho(\slaD_{\slaS_1},\slaD_{\slaS_0})(\lambda)=\rho(\slaD_{\slaS_{N_1}})(\lambda)-\rho(\slaD_{\slaS_{N_0}})(\lambda)$.
\end{proposition}

\begin{proof}
The bundles $F_{\lambda^\pm}^N$ can be pulled back to produce flat bundles $F_{\lambda^\pm}^{N\#N}$ over the connected sum $N\#N$. From the way $g_0$ and $g_1$ are constructed as connected sums, one can choose $g_0'$ PSC-isotopy to $g_0$ and $g_1'$ PSC-isotopy to $g_1$, such that $g_0'$ and $g_1'$ coincide on $\mathring{N}'$, and they are product metrics on $S\times I$ in view of \eqref{E:conn-sum}.
By Proposition~\ref{P:rel-rho cobor} and \eqref{E:rel-eta prop}, the relative rho invariant is unchanged when replacing $g_0$ and $g_1$ by $g_0'$ and $g_1'$. So we can just assume that $g_0$ and $g_1$ satisfy the properties of $g_0'$ and $g_1'$.
Using the gluing formula of the relative eta invariant (Corollary~\ref{C:splitting}, see Remark~\ref{R:bounded geom}),
\[
\eta(\slaD_{\slaS_1\otimes F_{\lambda^\pm}},\slaD_{\slaS_0\otimes F_{\lambda^\pm}})=\eta(\slaD_{\slaS_{N_1\#N_0}\otimes F_{\lambda^\pm}^{N\#N}})-\eta(\slaD_{\slaS_{N_0\#N_0}\otimes F_{\lambda^\pm}^{N\#N}}).
\]
Hence
\begin{equation}\label{E:rel-rho conn-sum-1}
\rho(\slaD_{\slaS_1},\slaD_{\slaS_0})(\lambda)=\rho(\slaD_{\slaS_{N_1\#N_0}})(\lambda)-\rho(\slaD_{\slaS_{N_0\#N_0}})(\lambda).
\end{equation}

On the other hand, the right-hand side of \eqref{E:rel-rho conn-sum-1} involves only compact manifolds. In this case it is known that $N_0\#N_0$ (resp. $N_1\#N_0$) is $\pi_1$-PSC-cobordant to the disjoint union $N_0\sqcup N_0$ (resp. $N_1\sqcup N_0$) (cf. Gajer \cite{Gajer87}, Carr \cite{Carr88}). It can be deduced by applying the APS index theorem to the PSC-cobordism between $N_0\#N_0$ and $N_0\sqcup N_0$ (twisted by the flat bundles corresponding to $\lambda^\pm$) that
\[
\eta(\slaD_{\slaS_{N_0\#N_0}\otimes F_{\lambda^+}^{N\#N}})-\eta(\slaD_{\slaS_{N_0\#N_0}\otimes F_{\lambda^-}^{N\#N}})=2\eta(\slaD_{\slaS_{N_0}\otimes F_{\lambda^+}^N})-2\eta(\slaD_{\slaS_{N_0}\otimes F_{\lambda^-}^N}),
\]
that is,
\begin{equation}\label{E:rel-rho conn-sum-2}
\rho(\slaD_{\slaS_{N_0\#N_0}})(\lambda)=2\rho(\slaD_{\slaS_{N_0}})(\lambda).
\end{equation}
Similarly,
\[
\begin{aligned}
\eta(\slaD_{\slaS_{N_1\#N_0}\otimes F_{\lambda^+}^{N\#N}})-\eta(\slaD_{\slaS_{N_1\#N_0}\otimes F_{\lambda^-}^{N\#N}})& =\eta(\slaD_{\slaS_{N_1}\otimes F_{\lambda^+}^N})-\eta(\slaD_{\slaS_{N_1}\otimes F_{\lambda^-}^N}) \\
& \quad+\eta(\slaD_{\slaS_{N_0}\otimes F_{\lambda^+}^N})-\eta(\slaD_{\slaS_{N_0}\otimes F_{\lambda^-}^N}),
\end{aligned}
\]
that is,
\begin{equation}\label{E:rel-rho conn-sum-3}
\rho(\slaD_{\slaS_{N_1\#N_0}})(\lambda)=\rho(\slaD_{\slaS_{N_1}})(\lambda)+\rho(\slaD_{\slaS_{N_0}})(\lambda).
\end{equation}
The proposition then follows from \eqref{E:rel-rho conn-sum-1}, \eqref{E:rel-rho conn-sum-2} and \eqref{E:rel-rho conn-sum-3}.
\end{proof}

\subsection{Space of PSC metrics on connected sums: Odd-dimensional case}\label{SS:PSC conn-sum odd}

On closed manifolds, the relative rho invariant is just the difference of two individual rho invariants. And Proposition~\ref{P:rel-rho cobor} holds obviously by the classical APS index theorem. Therefore, the rho invariant can be used to distinguish non-PSC-cobordant metrics on \emph{odd-dimensional} spin manifolds (cf. Remark~\ref{R:rel eta vanish}). One of the earliest results in this direction is due to Botvinnik and Gilkey \cite{BotviGilkey95eta-psc}. On manifolds satisfying certain conditions, by constructing a countable family of PSC metrics with distinct rho invariant values, they are able to prove the following stronger result than that of Proposition~\ref{P:non-iso 4m-1}.

\begin{theorem}[Botvinnik--Gilkey \cite{BotviGilkey95eta-psc}]\label{T:non-cobor cpt}
Let $N$ be a closed spin manifold of odd dimension $n\ge5$ with non-trivial finite fundamental group $\pi$, admitting a metric of positive scalar curvature. When $n\equiv1$ mod 4, assume also that $\pi$ has a non-zero virtual unitary representation $\lambda$ of virtual dimension 0 such that $\Tr\lambda(\alpha)=-\Tr\lambda(\alpha^{-1})$ for all $\alpha\in\pi$. Then $N$ admits infinite number of PSC metrics which are non-$\pi_1$-PSC-cobordant. In particular, $\pi_0(\frakR^+(N))$ is infinite, where $\frakR^+(N)$ denotes the space of PSC metrics on $N$.
\end{theorem}

From the properties of relative rho invariants discussed in Subsection \ref{SS:rel-rho}, we can generalize this theorem to non-compact situation.

\begin{theorem}\label{T:PSC conn-sum odd}
Let $N$ be as in Theorem~\ref{T:non-cobor cpt} and $N'$ be a non-compact spin manifold without boundary of the same dimension. Suppose $N'$ admits a complete uniformly PSC metric $h'$ of bounded geometry. Let $M=N\#N'$. Then there are infinite number of uniformly PSC metrics which are non-$\pi_1$-PSC-cobordant in $\frakR_\infty^+(M,h')$. In particular, $\pi_0(\frakR_\infty^+(M,h'))$ is infinite.
\end{theorem}

\begin{proof}
By Botvinnik--Gilkey's proof of Theorem~\ref{T:non-cobor cpt}, there exist infinitely many PSC metrics $h_i$ on $N$ and a representation $\lambda\in R_0(\pi)$ such that
\[
\rho(\slaD_{\slaS_{N_i}})(\lambda)\ne\rho(\slaD_{\slaS_{N_j}})(\lambda)
\]
for $i\ne j$. Here $N_i$ represents $N$ endowed with the metric $h_i$. These metrics are consequently non-$\pi_1$-PSC-cobordant in $\frakR^+(N)$. Put $g_i=h_i\# h'$. Then by Proposition~\ref{P:rel-rho conn-sum}, $\rho(\slaD_{\slaS_i},\slaD_{\slaS_j})(\lambda)\ne0$ for $i\ne j$. Hence by Proposition~\ref{P:rel-rho cobor}, $g_i$ and $g_j$ are not $\pi_1$-PSC-cobordant in $\frakR_\infty^+(M,h')$ when $i\ne j$, and the theorem follows.
\end{proof}

\begin{remark}\label{R:PSC conn-sum odd}
At first glance, Theorem~\ref{T:PSC conn-sum odd} may be proved in a simpler way by directly showing that $g_i$ and $g_j$ constructed above being PSC-isotopic (or PSC-cobordant) in $\frakR_\infty^+(M,h')$ implies that $h_i$ and $h_j$ are PSC-isotopic (or PSC-cobordant) in $\frakR^+(N)$. But we point out that this is not easy to do because a metric in a PSC-isotopy between $g_i$ and $g_j$ might not be constructed from a connected sum. Actually, such a metric could be different from $h'$ on a substantial (compact) subset of $N'$.
\end{remark}

\subsection{Space of PSC metrics on connected sums: Even-dimensional case}\label{SS:PSC conn-sum even}

In \cite[Theorem 9.2]{MRS16}, Mrowka, Ruberman, and Saveliev generalize Botvinnik--Gilkey's theorem to even-dimensional manifolds using their index theorem for end-periodic operators. To be precise, consider $N\times S^1$, where $N$ satisfies the conditions in Theorem~\ref{T:non-cobor cpt}. They show that if $h_i$ and $h_j$ are two PSC metrics on $N$ such that $\rho(\slaD_{\slaS_{N_i}})(\lambda)\ne\rho(\slaD_{\slaS_{N_j}})(\lambda)$ for $\lambda\in R_0(\pi)$, then the product metrics $h_i+ds^2$ and $h_j+ds^2$ are non-isotopic in the space of PSC metrics on $N\times S^1$. Combining their idea and our criteria resulted from the index formula, we can extend the result to non-compact case, which is an even-dimensional analogue of Proposition~\ref{P:non-iso 4m-1}.

\begin{theorem}\label{T:PSC conn-sum even}
Let $N$ be as in Theorem~\ref{T:non-cobor cpt}. Let $N'$ and $X$ be two non-compact spin manifolds without boundary such that $\dim N'=\dim N=\dim X-1$. Suppose $N'$ (resp. $X$) admits a complete uniformly PSC metric $h'$ (resp. $\gamma$) of bounded geometry. 
\begin{enumerate}
\item Let $M_1=(N\#N')\times S^1$. Then $\pi_0(\frakR_\infty^+(M_1,h'+ds^2))$ is infinite.

\item Let $M_2=(N\times S^1)\#X$. Then $\pi_0(\frakR_\infty^+(M_2,\gamma))$ is infinite.
\end{enumerate}
\end{theorem}

\begin{proof}
(\romannumeral1) Let $Z=(N\#N')\times[0,1]$. For an integer $k>0$, one can glue $3k+1$ copies of $Z$ to form a manifold
\[
W=\bigcup\limits_{i=-k}^{2k}Z_i,\qquad\mbox{where }Z_i\cong Z.
\]
The way of gluing is to identify the end $(N\#N')\times\{1\}$ of $Z_i$ with the end $(N\#N')\times\{0\}$ of $Z_{i+1}$.
Given two PSC-isotopic metrics $\gamma_0,\gamma_1\in\frakR_\infty^+(M_1,h'+ds^2)$, by the construction in \cite[Proof of Theorem 9.1]{MRS16}, for $k$ large enough, one can obtain a uniformly PSC metric $g_W$ on $W$ such that $g_W=\gamma_0$ on $Z_i$ for $i\le0$ and $g_W=\gamma_1$ on $Z_i$ for $i\ge k$. The copies of $Z$ in the middle are endowed with the metrics that lie in a PSC-isotopy between $\gamma_0$ and $\gamma_1$. One also needs to modify the metrics near the ends of $Z_i$ to make sure they match up.

Let $h_i\in\frakR^+(N)$ be as in the proof of Theorem~\ref{T:PSC conn-sum odd}. Put $\gamma_i=g_i+ds^2$, where $g_i=h_i\#h'$. Then for $i\ne j$, $\gamma_i$ and $\gamma_j$ are not PSC-isotopic. Otherwise, by the fact that they are both product metrics, in this case, the manifold $W$ constructed above can be seen as a PSC-cobordism between $g_i$ and $g_j$, which contradicts the fact in Theorem~\ref{T:PSC conn-sum odd}. The assertion thus follows.

(\romannumeral2) Let $Z=(N\times[0,1])\#X$, where the connected sum is performed in the interior of $N\times[0,1]$. Put $\gamma_i=(h_i+ds^2)\#\gamma$ and $\gamma_j=(h_j+ds^2)\#\gamma$. As in (\romannumeral1), if $\gamma_i$ and $\gamma_j$ are PSC-isotopic in $\frakR_\infty^+(M_2,\gamma)$, then one has a uniformly PSC metric $g_W$ on $W=\cup_{i=-k}^{2k}Z_i$ for some large $k$ such that $g_W$ restricts to $h_i$ and $h_j$ on the two boundary components of $W$ respectively. 

Note that here $W$ is not a PSC-cobordism between $h_i$ and $h_j$ as defined in Definition~\ref{D:equiv rela psc}, because $X$ is non-compact. However, we can use a cut-and-glue trick to replace $X$ by a compact manifold. As a first step, on each copy of $Z$, we fix a common hypersphere where $N\times[0,1]$ and $X$ takes connected sum. Near each copy of the hypersphere, we deform the metric to make it a product. It should be pointed out that on the copies of $Z$ in the middle of $W$, the new metric is not necessarily of uniformly PSC. Now we cut $W$ along each copy of the hypersphere. Since $k$ is finite, this results in cutting $\widetilde{W}$ (the elongation of $W$) along a \emph{closed} hypersurface $\Sigma$ so that
\[
\widetilde{W}=\widetilde{W}'\cup_\Sigma X',
\]
where $X'$ denotes $3k+1$ copies of $X$ with a ball removed. Let $\widetilde{W}_{\rm cpt}$ be $3k+1$ copies of $(N\times S^1)\#(N\times S^1)$ such that the metric on each copy agrees with that on the corresponding copy of $Z$ near the hypersphere. Again cut $\widetilde{W}_{\rm cpt}$ along $\Sigma$ so that
\[
\widetilde{W}_{\rm cpt}=K'\cup_\Sigma K''.
\]
Form two new manifolds
\[
\widetilde{W}_1=\widetilde{W}'\cup_\Sigma K'',\qquad \widetilde{W}_2=K'\cup_\Sigma X'.
\]
By the relative index theorem formulated by Bunke \cite[Theorem~1.14]{Bunke95},
\[
\{\widetilde{W}\}+\{\widetilde{W}_{\rm cpt}\}=\{\widetilde{W}_1\}+\{\widetilde{W}_2\},
\]
where $\{\cdot\}$ denotes the index of the spin Dirac operator on the corresponding space. Note that by the homotopy invariance of the index, the index on $\widetilde{W}$ with the deformed metric is the same as the index with the original metric. It is also clear that the metrics on $\widetilde{W}_{\rm cpt}$ and $\widetilde{W}_2$ can be compactly deformed back to be metrics of uniformly PSC, so $\{\widetilde{W}_{\rm cpt}\}=\{\widetilde{W}_2\}=0$. This reduces the index on $\widetilde{W}$ to that on $\widetilde{W}_1$, which is actually given by the classical compact-version of APS index formula. One can again use the contradiction argument as before. The proof is completed.
\end{proof}

\begin{remark}\label{R:PSC conn-sum}
If $\frakD_\infty(M)$ denotes the group of spin structure preserving diffeomorphisms of $M$ which are identity at infinity and let $\frakM_\infty^+(M,g)=\frakR_\infty^+(M,g)/\frakD_\infty(M)$ be the corresponding \emph{moduli space}. Then the conclusions of Theorems~\ref{T:PSC conn-sum odd} and \ref{T:PSC conn-sum even} hold with $\frakR_\infty^+$ replaced by $\frakM_\infty^+$.
\end{remark}

\subsection{Space of PSC metrics on manifolds with boundary}\label{SS:PSC bdry}

The idea of Subsections \ref{SS:PSC conn-sum odd} and \ref{SS:PSC conn-sum even} can be used to obtain similar conclusions about the space of PSC metrics on the connected sum of a closed manifold and a manifold with boundary. Using the notation as in Remark~\ref{R:conn-sum}, we have the following.

\begin{proposition}\label{P:PSC bdry}
Let $N$ be as in Theorem~\ref{T:non-cobor cpt}. Let $N'$ and $X$ be two compact Riemannian spin manifolds with boundary such that $\dim N'=\dim N=\dim X-1$. Suppose $N'$ (resp. $X$) admits a PSC metric $h'$ (resp. $\gamma$) which is the restriction of a uniformly PSC metric of bounded geometry on a Riemannian spin manifold of the same dimension without boundary. Set
\[
M_0=N\#N',\quad M_1=(N\#N')\times S^1,\quad M_2=(N\times S^1)\#X.
\]
Then, the spaces $\frakR^+_{\pM_0}(M_0,h')$, $\frakR^+_{\pM_1}(M_1,h'+ds^2)$ and $\frakR^+_{\pM_2}(M_2,\gamma)$ all have infinitely many path components.
\end{proposition}

\begin{proof}
Assume $h'$ is the restriction of $\tilde{h}'$, where $\tilde{h}'$ is a uniformly PSC metric of bounded geometry on a spin manifold $\widetilde{N}'$ of the same dimension as $N'$ without boundary. We construct $\tilde{g}_i=h_i\#\tilde{h}'\in\frakR^+_\infty(N\#\widetilde{N}',\tilde{h'})$, where $\{h_i\}$ consists of infinitely many non-PSC-isotopic metrics in $\frakR^+(N)$ as in the proof of Theorem~\ref{T:PSC conn-sum odd}. Let $g_i$ denote the restriction of $\tilde{g}_i$ to $M_0$. By the above construction, we can certainly require that each $g_i$ belongs to $\frakR^+_{\pM_0}(M_0,h')$. For $i\ne j$, since $\tilde{g}_i$ and $\tilde{g}_j$ belong to different path components of $\frakR^+_\infty(N\#\widetilde{N}',\tilde{h'})$, it follows easily that $g_i$ and $g_j$ belong to different path components of $\frakR^+_{\pM_0}(M_0,h')$. Therefore $\frakR^+_{\pM_0}(M_0,h')$ has infinitely many path components.

Using Theorem~\ref{T:PSC conn-sum even}, the assertion for $\frakR^+_{\pM_1}(M_1,h'+ds^2)$ and $\frakR^+_{\pM_2}(M_2,\gamma)$ can be proved in a similar pattern.
\end{proof}

\begin{remark}\label{R:PSC bdry}
The space of PSC metrics on a compact manifold with boundary has been investigated in more general settings by Botvinnik--Ebert--Randal-Williams \cite{BotviEbertRandal17}, Ebert--Randal-Williams \cite{EbertRandal19} and Cecchini--Seyedhosseini--Zenobi \cite{CeccSeyeZeno23}. More precisely, \cite{BotviEbertRandal17,EbertRandal19} address the non-triviality of the homotopy groups of the space assuming the metric is a product near the boundary and restricts to a fixed one on the boundary. \cite{CeccSeyeZeno23} shows that the space has infinite many path components assuming only the metric is a product near the boundary. This means the boundary must admit a PSC metric. In contrast, our result allows the metric to be of non-product type near the boundary.
\end{remark}

\begin{remark}\label{R:PSC bdry-1}
One situation that the hypothesis of Proposition~\ref{P:PSC bdry} holds is that when $h'$ (resp. $\gamma$) can be extended to a PSC metric on the \emph{double} of $N'$ (resp. $X$). By recent results of B\"ar--Hanke \cite{BaerHanke23} (see also de Almeida \cite{Almeida85} and Rosenberg--Weinberg \cite{RosenWein23}), this can be achieved if there exists a PSC metric with non-negative mean curvature along the boundary.\footnote{In fact, this is always true regardless of the manifold being spin or not.} In this perspective, our result can be formulated more concretely.
\end{remark}

\begin{theorem}\label{T:PSC bdry}
Let $N$ be as in Theorem~\ref{T:non-cobor cpt}. Let $N'$ and $X$ be two compact Riemannian spin manifolds with boundary such that $\dim N'=\dim N=\dim X-1$. Suppose $N'$ (resp. $X$) admits a PSC metric $h'$ (resp. $\gamma$) with non-negative mean curvature along $\p N'$ (resp. $\p X$). Again set
\[
M_0=N\#N',\quad M_1=(N\#N')\times S^1,\quad M_2=(N\times S^1)\#X.
\]
Then, the spaces $\frakR^+_{\pM_0}(M_0,h')$, $\frakR^+_{\pM_1}(M_1,h'+ds^2)$ and $\frakR^+_{\pM_2}(M_2,\gamma)$ all have infinitely many path components.
\end{theorem}


\bibliographystyle{amsplain}

\begin{bibdiv}
\begin{biblist}

\bib{APS1}{article}{
      author={Atiyah, M.~F.},
      author={Patodi, V.~K.},
      author={Singer, I.~M.},
       title={Spectral asymmetry and {R}iemannian geometry. {I}},
        date={1975},
        ISSN={0305-0041},
     journal={Math. Proc. Cambridge Philos. Soc.},
      volume={77},
       pages={43\ndash 69},
         url={https://doi.org/10.1017/S0305004100049410},
      review={\MR{397797}},
}

\bib{AtSinger63}{article}{
      author={Atiyah, M.~F.},
      author={Singer, I.~M.},
       title={The index of elliptic operators on compact manifolds},
        date={1963},
        ISSN={0002-9904},
     journal={Bull. Amer. Math. Soc.},
      volume={69},
       pages={422\ndash 433},
         url={https://doi.org/10.1090/S0002-9904-1963-10957-X},
      review={\MR{157392}},
}

\bib{BamlerKleiner19}{article}{
      author={Bamler, R.},
      author={Kleiner, B.},
       title={Ricci flow and contractibility of spaces of metrics},
        date={2019},
      eprint={https://arxiv.org/abs/1909.08710},
         url={https://arxiv.org/abs/1909.08710},
}

\bib{BMR18}{article}{
      author={Bandara, L.},
      author={McIntosh, A.},
      author={Ros\'{e}n, A.},
       title={Riesz continuity of the {A}tiyah-{S}inger {D}irac operator under
  perturbations of the metric},
        date={2018},
        ISSN={0025-5831},
     journal={Math. Ann.},
      volume={370},
      number={1-2},
       pages={863\ndash 915},
         url={https://doi.org/10.1007/s00208-017-1610-7},
      review={\MR{3747505}},
}

\bib{BaerBallmann12}{incollection}{
      author={B\"{a}r, C.},
      author={Ballmann, W.},
       title={Boundary value problems for elliptic differential operators of
  first order},
        date={2012},
   booktitle={Surveys in differential geometry. {V}ol. {XVII}},
      series={Surv. Differ. Geom.},
      volume={17},
   publisher={Int. Press, Boston, MA},
       pages={1\ndash 78},
         url={https://doi.org/10.4310/SDG.2012.v17.n1.a1},
      review={\MR{3076058}},
}

\bib{BaerBallmann16}{incollection}{
      author={B\"{a}r, C.},
      author={Ballmann, W.},
       title={Guide to elliptic boundary value problems for {D}irac-type
  operators},
        date={2016},
   booktitle={Arbeitstagung {B}onn 2013},
      series={Progr. Math.},
      volume={319},
   publisher={Birkh\"{a}user/Springer, Cham},
       pages={43\ndash 80},
         url={https://doi.org/10.1007/978-3-319-43648-7_3},
      review={\MR{3618047}},
}

\bib{BaerHanke23}{incollection}{
      author={B\"{a}r, C.},
      author={Hanke, B.},
       title={Boundary conditions for scalar curvature},
        date={[2023] \copyright 2023},
   booktitle={Perspectives in scalar curvature. {V}ol. 2},
   publisher={World Sci. Publ., Hackensack, NJ},
       pages={325\ndash 377},
      review={\MR{4577919}},
}

\bib{BeleHu15}{article}{
      author={Belegradek, I.},
      author={Hu, J.},
       title={Connectedness properties of the space of complete nonnegatively
  curved planes},
        date={2015},
        ISSN={0025-5831,1432-1807},
     journal={Math. Ann.},
      volume={362},
      number={3-4},
       pages={1273\ndash 1286},
         url={https://doi.org/10.1007/s00208-014-1159-7},
      review={\MR{3368099}},
}

\bib{BeGeVe}{book}{
      author={Berline, N.},
      author={Getzler, E.},
      author={Vergne, M.},
       title={Heat kernels and {D}irac operators},
      series={Grundlehren Math. Wiss. [Fundamental Principles of Mathematical
  Sciences]},
   publisher={Springer-Verlag, Berlin},
        date={1992},
      volume={298},
        ISBN={3-540-53340-0},
         url={https://doi.org/10.1007/978-3-642-58088-8},
      review={\MR{1215720}},
}

\bib{BBMM21}{article}{
      author={Bessi\`eres, L.},
      author={Besson, G.},
      author={Maillot, S.},
      author={Marques, F.},
       title={Deforming 3-manifolds of bounded geometry and uniformly positive
  scalar curvature},
        date={2021},
        ISSN={1435-9855,1435-9863},
     journal={J. Eur. Math. Soc. (JEMS)},
      volume={23},
      number={1},
       pages={153\ndash 184},
         url={https://doi.org/10.4171/jems/1008},
      review={\MR{4186465}},
}

\bib{BoossLeschPhillips05}{article}{
      author={Booss-Bavnbek, B.},
      author={Lesch, M.},
      author={Phillips, J.},
       title={Unbounded {F}redholm operators and spectral flow},
        date={2005},
        ISSN={0008-414X,1496-4279},
     journal={Canad. J. Math.},
      volume={57},
      number={2},
       pages={225\ndash 250},
         url={https://doi.org/10.4153/CJM-2005-010-1},
      review={\MR{2124916}},
}

\bib{BotviEbertRandal17}{article}{
      author={Botvinnik, B.},
      author={Ebert, J.},
      author={Randal-Williams, O.},
       title={Infinite loop spaces and positive scalar curvature},
        date={2017},
        ISSN={0020-9910},
     journal={Invent. Math.},
      volume={209},
      number={3},
       pages={749\ndash 835},
         url={https://doi.org/10.1007/s00222-017-0719-3},
      review={\MR{3681394}},
}

\bib{BotviGilkey95eta-psc}{article}{
      author={Botvinnik, B.},
      author={Gilkey, P.~B.},
       title={The eta invariant and metrics of positive scalar curvature},
        date={1995},
        ISSN={0025-5831},
     journal={Math. Ann.},
      volume={302},
      number={3},
       pages={507\ndash 517},
         url={https://doi.org/10.1007/BF01444505},
      review={\MR{1339924}},
}

\bib{BotviGilkey96psc-spaceform}{article}{
      author={Botvinnik, B.},
      author={Gilkey, P.~B.},
       title={Metrics of positive scalar curvature on spherical space forms},
        date={1996},
        ISSN={0008-414X},
     journal={Canad. J. Math.},
      volume={48},
      number={1},
       pages={64\ndash 80},
         url={https://doi.org/10.4153/CJM-1996-003-0},
      review={\MR{1382476}},
}

\bib{BG92spinor}{article}{
      author={Bourguignon, J.-P.},
      author={Gauduchon, P.},
       title={Spineurs, op\'{e}rateurs de {D}irac et variations de
  m\'{e}triques},
        date={1992},
        ISSN={0010-3616},
     journal={Comm. Math. Phys.},
      volume={144},
      number={3},
       pages={581\ndash 599},
         url={http://projecteuclid.org/euclid.cmp/1104249410},
      review={\MR{1158762}},
}

\bib{BrShi21-2}{article}{
      author={Braverman, M.},
      author={Shi, P.},
       title={An {APS} index theorem for even-dimensional manifolds with
  non-compact boundary},
        date={2021},
        ISSN={1019-8385},
     journal={Comm. Anal. Geom.},
      volume={29},
      number={2},
       pages={293\ndash 327},
         url={https://doi.org/10.4310/CAG.2021.v29.n2.a2},
      review={\MR{4250324}},
}

\bib{BrShi21}{article}{
      author={Braverman, M.},
      author={Shi, P.},
       title={The {A}tiyah-{P}atodi-{S}inger index on manifolds with
  non-compact boundary},
        date={2021},
        ISSN={1050-6926},
     journal={J. Geom. Anal.},
      volume={31},
      number={4},
       pages={3713\ndash 3763},
         url={https://doi.org/10.1007/s12220-020-00412-3},
      review={\MR{4236541}},
}

\bib{BruningLesch99}{article}{
      author={Br\"{u}ning, J.},
      author={Lesch, M.},
       title={On the {$\eta$}-invariant of certain nonlocal boundary value
  problems},
        date={1999},
        ISSN={0012-7094},
     journal={Duke Math. J.},
      volume={96},
      number={2},
       pages={425\ndash 468},
         url={https://doi.org/10.1215/S0012-7094-99-09613-8},
      review={\MR{1666570}},
}

\bib{Bunke92}{article}{
      author={Bunke, U.},
       title={Relative index theory},
        date={1992},
        ISSN={0022-1236},
     journal={J. Funct. Anal.},
      volume={105},
      number={1},
       pages={63\ndash 76},
         url={https://doi.org/10.1016/0022-1236(92)90072-Q},
      review={\MR{1156670}},
}

\bib{Bunke93comparison}{inproceedings}{
      author={Bunke, U.},
       title={Comparison of {D}irac operators on manifolds with boundary},
        date={1993},
   booktitle={Proceedings of the {W}inter {S}chool ``{G}eometry and {P}hysics''
  ({S}rn\'{\i}, 1991)},
       pages={133\ndash 141},
      review={\MR{1246627}},
}

\bib{Bunke95}{article}{
      author={Bunke, U.},
       title={A {$K$}-theoretic relative index theorem and {C}allias-type
  {D}irac operators},
        date={1995},
        ISSN={0025-5831},
     journal={Math. Ann.},
      volume={303},
      number={2},
       pages={241\ndash 279},
         url={https://doi.org/10.1007/BF01460989},
      review={\MR{1348799}},
}

\bib{Carlotto21}{article}{
      author={Carlotto, A.},
       title={A survey on positive scalar curvature metrics},
        date={2021},
        ISSN={1972-6724},
     journal={Boll. Unione Mat. Ital.},
      volume={14},
      number={1},
       pages={17\ndash 42},
         url={https://doi.org/10.1007/s40574-020-00228-7},
      review={\MR{4218620}},
}

\bib{Carr88}{article}{
      author={Carr, R.},
       title={Construction of manifolds of positive scalar curvature},
        date={1988},
        ISSN={0002-9947},
     journal={Trans. Amer. Math. Soc.},
      volume={307},
      number={1},
       pages={63\ndash 74},
         url={https://doi.org/10.2307/2000751},
      review={\MR{936805}},
}

\bib{CeccSeyeZeno23}{article}{
      author={Cecchini, S.},
      author={Seyedhosseini, M.},
      author={Zenobi, V.},
       title={Relative torsion and bordism classes of positive scalar curvature
  metrics on manifolds with boundary},
        date={2023},
        ISSN={0025-5874},
     journal={Math. Z.},
      volume={305},
      number={3},
       pages={Paper No. 36, 19},
         url={https://doi.org/10.1007/s00209-023-03334-2},
      review={\MR{4648692}},
}

\bib{Almeida85}{article}{
      author={de~Almeida, S.},
       title={Minimal hypersurfaces of a positive scalar curvature manifold},
        date={1985},
        ISSN={0025-5874,1432-1823},
     journal={Math. Z.},
      volume={190},
      number={1},
       pages={73\ndash 82},
         url={https://doi.org/10.1007/BF01159165},
      review={\MR{793350}},
}

\bib{EbertRandal19}{article}{
      author={Ebert, J.},
      author={Randal-Williams, O.},
       title={Infinite loop spaces and positive scalar curvature in the
  presence of a fundamental group},
        date={2019},
        ISSN={1465-3060},
     journal={Geom. Topol.},
      volume={23},
      number={3},
       pages={1549\ndash 1610},
         url={https://doi.org/10.2140/gt.2019.23.1549},
      review={\MR{3956897}},
}

\bib{Gajer87}{article}{
      author={Gajer, P.},
       title={Riemannian metrics of positive scalar curvature on compact
  manifolds with boundary},
        date={1987},
        ISSN={0232-704X},
     journal={Ann. Global Anal. Geom.},
      volume={5},
      number={3},
       pages={179\ndash 191},
         url={https://doi.org/10.1007/BF00128019},
      review={\MR{962295}},
}

\bib{Getzler93}{article}{
      author={Getzler, E.},
       title={The odd {C}hern character in cyclic homology and spectral flow},
        date={1993},
        ISSN={0040-9383},
     journal={Topology},
      volume={32},
      number={3},
       pages={489\ndash 507},
         url={https://doi.org/10.1016/0040-9383(93)90002-D},
      review={\MR{1231957}},
}

\bib{Gilkey95book}{book}{
      author={Gilkey, P.~B.},
       title={Invariance theory, the heat equation, and the {A}tiyah-{S}inger
  index theorem},
     edition={Second},
      series={Stud. Adv. Math.},
   publisher={CRC Press, Boca Raton, FL},
        date={1995},
        ISBN={0-8493-7874-4},
      review={\MR{1396308}},
}

\bib{Gromov23Four}{incollection}{
      author={Gromov, M.},
       title={Four lectures on scalar curvature},
        date={[2023] \copyright 2023},
   booktitle={Perspectives in scalar curvature. {V}ol. 1},
   publisher={World Sci. Publ., Hackensack, NJ},
       pages={1\ndash 514},
      review={\MR{4577903}},
}

\bib{GromovLawson80classification}{article}{
      author={Gromov, M.},
      author={Lawson, H.~B., Jr.},
       title={The classification of simply connected manifolds of positive
  scalar curvature},
        date={1980},
        ISSN={0003-486X},
     journal={Ann. of Math. (2)},
      volume={111},
      number={3},
       pages={423\ndash 434},
         url={https://doi.org/10.2307/1971103},
      review={\MR{577131}},
}

\bib{GromovLawson80spin}{article}{
      author={Gromov, M.},
      author={Lawson, H.~B., Jr.},
       title={Spin and scalar curvature in the presence of a fundamental group.
  {I}},
        date={1980},
        ISSN={0003-486X},
     journal={Ann. of Math. (2)},
      volume={111},
      number={2},
       pages={209\ndash 230},
         url={https://doi.org/10.2307/1971198},
      review={\MR{569070}},
}

\bib{GromovLawson83}{article}{
      author={Gromov, M.},
      author={Lawson, H.~B., Jr.},
       title={Positive scalar curvature and the {D}irac operator on complete
  {R}iemannian manifolds},
        date={1983},
        ISSN={0073-8301},
     journal={Inst. Hautes \'{E}tudes Sci. Publ. Math.},
      number={58},
       pages={83\ndash 196 (1984)},
         url={http://www.numdam.org/item?id=PMIHES_1983__58__83_0},
      review={\MR{720933}},
}

\bib{Hitchin74}{article}{
      author={Hitchin, N.},
       title={Harmonic spinors},
        date={1974},
        ISSN={0001-8708},
     journal={Advances in Math.},
      volume={14},
       pages={1\ndash 55},
         url={https://doi.org/10.1016/0001-8708(74)90021-8},
      review={\MR{358873}},
}

\bib{KirkLesch04}{article}{
      author={Kirk, P.},
      author={Lesch, M.},
       title={The {$\eta$}-invariant, {M}aslov index, and spectral flow for
  {D}irac-type operators on manifolds with boundary},
        date={2004},
        ISSN={0933-7741},
     journal={Forum Math.},
      volume={16},
      number={4},
       pages={553\ndash 629},
         url={https://doi.org/10.1515/form.2004.027},
      review={\MR{2044028}},
}

\bib{KreckStolz93}{article}{
      author={Kreck, M.},
      author={Stolz, S.},
       title={Nonconnected moduli spaces of positive sectional curvature
  metrics},
        date={1993},
        ISSN={0894-0347},
     journal={J. Amer. Math. Soc.},
      volume={6},
      number={4},
       pages={825\ndash 850},
         url={https://doi.org/10.2307/2152742},
      review={\MR{1205446}},
}

\bib{LawMic89}{book}{
      author={Lawson, H.~B., Jr.},
      author={Michelsohn, M.-L.},
       title={Spin geometry},
      series={Princeton Math. Ser.},
   publisher={Princeton University Press, Princeton, NJ},
        date={1989},
      volume={38},
        ISBN={0-691-08542-0},
      review={\MR{1031992}},
}

\bib{LiSuWang24}{article}{
      author={Li, Y.},
      author={Su, G.},
      author={Wang, X.},
       title={Spectral flow, {L}larull's rigidity theorem in odd dimensions and
  its generalization},
        date={2024},
        ISSN={1674-7283},
     journal={Sci. China Math.},
      volume={67},
      number={5},
       pages={1103\ndash 1114},
         url={https://doi.org/10.1007/s11425-023-2138-5},
      review={\MR{4739559}},
}

\bib{Lich63}{article}{
      author={Lichnerowicz, A.},
       title={Spineurs harmoniques},
        date={1963},
        ISSN={0001-4036},
     journal={C. R. Acad. Sci. Paris},
      volume={257},
       pages={7\ndash 9},
      review={\MR{156292}},
}

\bib{Llarull98}{article}{
      author={Llarull, M.},
       title={Sharp estimates and the {D}irac operator},
        date={1998},
        ISSN={0025-5831},
     journal={Math. Ann.},
      volume={310},
      number={1},
       pages={55\ndash 71},
         url={https://doi.org/10.1007/s002080050136},
      review={\MR{1600027}},
}

\bib{Marques12}{article}{
      author={Marques, F.},
       title={Deforming three-manifolds with positive scalar curvature},
        date={2012},
        ISSN={0003-486X},
     journal={Ann. of Math. (2)},
      volume={176},
      number={2},
       pages={815\ndash 863},
         url={https://doi.org/10.4007/annals.2012.176.2.3},
      review={\MR{2950765}},
}

\bib{MRS16}{article}{
      author={Mrowka, T.},
      author={Ruberman, D.},
      author={Saveliev, N.},
       title={An index theorem for end-periodic operators},
        date={2016},
        ISSN={0010-437X},
     journal={Compos. Math.},
      volume={152},
      number={2},
       pages={399\ndash 444},
         url={https://doi.org/10.1112/S0010437X15007502},
      review={\MR{3462557}},
}

\bib{PiazzaSchick07torsion-rho}{article}{
      author={Piazza, P.},
      author={Schick, T.},
       title={Groups with torsion, bordism and rho invariants},
        date={2007},
        ISSN={0030-8730},
     journal={Pacific J. Math.},
      volume={232},
      number={2},
       pages={355\ndash 378},
         url={https://doi.org/10.2140/pjm.2007.232.355},
      review={\MR{2366359}},
}

\bib{RobbinSalamon95}{article}{
      author={Robbin, J.},
      author={Salamon, D.},
       title={The spectral flow and the {M}aslov index},
        date={1995},
        ISSN={0024-6093,1469-2120},
     journal={Bull. London Math. Soc.},
      volume={27},
      number={1},
       pages={1\ndash 33},
         url={https://doi.org/10.1112/blms/27.1.1},
      review={\MR{1331677}},
}

\bib{RosenbergStolz01}{incollection}{
      author={Rosenberg, J.},
      author={Stolz, S.},
       title={Metrics of positive scalar curvature and connections with
  surgery},
        date={2001},
   booktitle={Surveys on surgery theory, {V}ol. 2},
      series={Ann. of Math. Stud.},
      volume={149},
   publisher={Princeton Univ. Press, Princeton, NJ},
       pages={353\ndash 386},
      review={\MR{1818778}},
}

\bib{RosenWein23}{article}{
      author={Rosenberg, J.},
      author={Weinberger, S.},
       title={Positive scalar curvature on manifolds with boundary and their
  doubles},
        date={2023},
        ISSN={1558-8599,1558-8602},
     journal={Pure Appl. Math. Q.},
      volume={19},
      number={6},
       pages={2919\ndash 2950},
         url={https://doi.org/10.4310/pamq.2023.v19.n6.a12},
      review={\MR{4703039}},
}

\bib{Ruberman98}{article}{
      author={Ruberman, D.},
       title={An obstruction to smooth isotopy in dimension {$4$}},
        date={1998},
        ISSN={1073-2780},
     journal={Math. Res. Lett.},
      volume={5},
      number={6},
       pages={743\ndash 758},
         url={https://doi.org/10.4310/MRL.1998.v5.n6.a5},
      review={\MR{1671187}},
}

\bib{Ruberman01}{article}{
      author={Ruberman, D.},
       title={Positive scalar curvature, diffeomorphisms and the
  {S}eiberg-{W}itten invariants},
        date={2001},
        ISSN={1465-3060},
     journal={Geom. Topol.},
      volume={5},
       pages={895\ndash 924},
         url={https://doi.org/10.2140/gt.2001.5.895},
      review={\MR{1874146}},
}

\bib{Shi22}{article}{
      author={Shi, P.},
       title={The relative eta invariant for a pair of {D}irac-type operators
  on non-compact manifolds},
        date={2022},
        ISSN={0022-2518},
     journal={Indiana Univ. Math. J.},
      volume={71},
      number={5},
       pages={1923\ndash 1966},
      review={\MR{4509824}},
}

\bib{TuschmannWraith15book}{book}{
      author={Tuschmann, W.},
      author={Wraith, D.~J.},
       title={Moduli spaces of {R}iemannian metrics},
      series={Oberwolfach Semin.},
   publisher={Birkh\"{a}user Verlag, Basel},
        date={2015},
      volume={46},
        ISBN={978-3-0348-0947-4; 978-3-0348-0948-1},
         url={https://doi.org/10.1007/978-3-0348-0948-1},
        note={Second corrected printing},
      review={\MR{3445334}},
}

\bib{Zhang01book}{book}{
      author={Zhang, W.},
       title={Lectures on {C}hern-{W}eil theory and {W}itten deformations},
      series={Nankai Tracts Math.},
   publisher={World Scientific Publishing Co., Inc., River Edge, NJ},
        date={2001},
      volume={4},
        ISBN={981-02-4686-2},
         url={https://doi.org/10.1142/9789812386588},
      review={\MR{1864735}},
}

\end{biblist}
\end{bibdiv}


\end{document}